\documentclass[oneside,article]{memoir}

\usepackage{amsmath}         
\usepackage{amsthm,thmtools} 
\usepackage{enumitem}        
\usepackage{appendix}

\usepackage{tikz}
\usetikzlibrary{matrix,arrows,decorations.pathreplacing,angles,quotes}
  
\usepackage{tikz-cd}

\usepackage[nobysame,alphabetic,initials]{amsrefs} 
\usepackage[hidelinks]{hyperref}                   
\usepackage[capitalize,nosort]{cleveref}           

\usepackage{amssymb}
\usepackage[mathscr]{euscript}
\usepackage{relsize}
\usepackage{stmaryrd}
\usepackage{stackengine}

\usepackage{indentfirst} 
\usepackage{float} 
\usepackage{multicol}
\usepackage{tabularx}

\usepackage[nomargin,inline,marginclue,draft]{fixme} 
  \fxsetup{targetlayout=color}

\isopage[12]

\setlrmargins{*}{*}{1}
\checkandfixthelayout

\counterwithout{section}{chapter}

  \newcommand{\Ho}{\operatorname{Ho}}

  \newcommand{\op}{{\mathord\mathrm{op}}}

  \newcommand{\colim}{\operatorname{colim}}

  \newcommand{\bd}{\partial}

  \newcommand{\simp}[1]{\mathord\Delta^{#1}}

  \makeatletter
  \def\horn#1{\expandafter\horn@i#1,,\@nil}
  \def\horn@i#1,#2,#3\@nil{\Lambda^{#1}_{#2}}
  \makeatother

  \renewcommand{\tilde}{\widetilde}

  \newcommand{\cat}[1]{\mathscr{#1}}

  \newcommand{\ncat}[1]{\mathsf{#1}}

  \newcommand{\Simp}{\Delta}

  \newcommand{\sSet}{\ncat{sSet}}

  \newcommand{\from}{\colon}

  \declaretheorem[style=definition,within=section]{definition}
  \declaretheorem[style=definition,numberlike=definition]{notation}
  \declaretheorem[style=definition,numberlike=definition]{example}
  \declaretheorem[style=definition,numberlike=definition]{remark}

  \declaretheorem[style=plain,numberlike=definition]{corollary}
  \declaretheorem[style=plain,numberlike=definition]{lemma}
  \declaretheorem[style=plain,numberlike=definition]{proposition}
  \declaretheorem[style=plain,numberlike=definition]{theorem}

  \declaretheorem[style=plain,numbered=no,name=Theorem]{theorem*}

  \Crefname{corollary}{Corollary}{Corollaries}
  \Crefname{definition}{Definition}{Definitions}
  \Crefname{notation}{Notation}{Notations}
  \Crefname{lemma}{Lemma}{Lemmas}
  \Crefname{proposition}{Proposition}{Propositions}
  \Crefname{remark}{Remark}{Remarks}
  \Crefname{theorem}{Theorem}{Theorems}
  \Crefname{construction}{Construction}{Constructions}
  \Crefname{example}{Example}{Examples}
  \Crefname{question}{Question}{Questions}

  \newlist{axioms}{enumerate}{1}
  \Crefname{axiomsi}{}{}

  \newenvironment{tikzeq*}
  {
    \begingroup
    \begin{equation*}
    \begin{tikzpicture}[baseline=(current bounding box.center)]
  }
  {
    \end{tikzpicture}
    \end{equation*}
    \endgroup
    \ignorespacesafterend
  }

  \tikzset
  {
    diagram/.style=
    {
      matrix of math nodes,
      column sep={4.3em,between origins},
      row sep={4em,between origins},
      text height=1.5ex,
      text depth=.25ex
    },
    over/.style={preaction={draw=white,-,line width=6pt}},
    every to/.style={font=\footnotesize},
    inj/.style={right hook->},
    surj/.style={-{Latex[open]}},
    cof/.style={>->},
    fib/.style={->>},
  }

  \DeclareFontFamily{U}{mathx}{\hyphenchar\font45}

  \DeclareFontShape{U}{mathx}{m}{n}{
    <5> <6> <7> <8> <9> <10>
    <10.95> <12> <14.4> <17.28> <20.74> <24.88>
    mathx10}{}

  \DeclareSymbolFont{mathx}{U}{mathx}{m}{n}

  \DeclareFontFamily{U}{mathb}{\hyphenchar\font45}

  \DeclareFontShape{U}{mathb}{m}{n}{
    <5> <6> <7> <8> <9> <10>
    <10.95> <12> <14.4> <17.28> <20.74> <24.88>
    mathb10}{}

  \DeclareSymbolFont{mathb}{U}{mathb}{m}{n}

  \DeclareMathAccent{\widebar}{0}{mathx}{"73}

  \DeclareMathSymbol{\Rsh}{\mathrel}{mathb}{"E9}

  \DeclareFontFamily{U}{MnSymbolA}{}

  \DeclareFontShape{U}{MnSymbolA}{m}{n}{
    <-6> MnSymbolA5
    <6-7> MnSymbolA6
    <7-8> MnSymbolA7
    <8-9> MnSymbolA8
    <9-10> MnSymbolA9
    <10-12> MnSymbolA10
    <12-> MnSymbolA12}{}

  \DeclareSymbolFont{MnSyA}{U}{MnSymbolA}{m}{n}

  \DeclareMathSymbol{\twoheaddownarrow}{\mathrel}{MnSyA}{27}

  \newcommand{\MSC}[1]{%
    \let\thempfn\relax
    \footnotetext[0]{2020 Mathematics Subject Classification: #1.}
  }

\newcommand{\cSet}[1][]{\mathsf{cSet}_{#1}} 
\newcommand{\Gpd}{\mathsf{Gpd}}

\newcommand{\Graph}{\mathsf{Graph}}
\newcommand{\Set}{\mathsf{Set}}

\newcommand{\V}[1]{{#1}_{V}} 
\newcommand{\E}[1]{{#1}_{E}} 

\newcommand{\sk}[1]{{\mathord\mathrm{sk}}_{#1}} 
\newcommand{\cosk}[1]{{\mathord\mathrm{cosk}}_{#1}} 

\newcommand{\Top}{\ncat{Top}}

\newcommand{\gtimes}{\ \square \ } 
\newcommand{\gexp}[2]{\ensuremath{#1}^{\square #2}} 

\newcommand{\image}{\operatorname{im}} 

\newcommand{\boxcat}{\mathord{\square}} 

\newcommand{\face}[2]{\partial^{#1}_{#2}} 
\newcommand{\degen}[2]{\sigma^{#1}_{#2}} 
\newcommand{\cube}[1]{\mathord{\square^{#1}}} 
\newcommand{\bdcube}[1]{\mathord{\bd\cube{#1}}} 
\newcommand{\obox}[2]{\mathord{\sqcap^{#1}_{#2}}} 
\newcommand{\Kan}{\mathsf{Kan}}
\newcommand{\gnerve}[1][]{\mathrm{N}_{#1}} 

\author{Krzysztof Kapulkin \and Udit Mavinkurve}
\title{Homotopy $n$-types of cubical sets and graphs}
\date{\today}

\begin{document}

  \maketitle

  \begin{abstract}
     We give a new construction of the model structure on the category of simplicial sets for homotopy $n$-types, originally due to Elvira-Donazar and Hernandez-Paricio, using a right transfer along the coskeleton functor.
     We observe that an analogous model structure can be constructed on the category of cubical sets, and use it to equip the category of (simple) graphs with a fibration category structure whose weak equivalences are discrete $n$-equivalences.
  \end{abstract}

  \setlist[enumerate]{label=(\arabic*)}

  
\section*{Introduction}

Discrete homotopy theory is a homotopy theory of graphs \cite{kramer-laubenbacher,babson-barcelo-longueville-laubenbacher,barcelo-laubenbacher,barcelo-kramer-laubenbacher-weaver}, in essence studying the combinatorial, rather than the topological, shape of graphs.
Its techniques have found numerous applications, both within and outside mathematics, including in: matroid theory, hyperplane arrangements, and data analysis.

A central open problem in the field is to determine whether the discrete homotopy theory of graphs and the classical homotopy theory of (topological) spaces are equivalent.
More precisely, the problem asks to show that the graph nerve functor of \cite{babson-barcelo-longueville-laubenbacher,carranza-kapulkin:cubical-graphs}, associating a cubical Kan complex to a graph, is a DK-equivalence of relative categories \cite{barwick-kan}.
The positive answer to this question, generally believed to be true by the experts, would allow the import of various techniques of classical homotopy theory into the discrete realm, e.g., the Mayer--Vietoris sequence or the Blakers--Massey theorem.

While interesting and broadly applicable, the problem has proven quite difficult.
In the present paper, we propose a new line of attack by establishing the theory of discrete homotopy $n$-types.
These are graphs whose non-trivial discrete homotopy groups are concentrated in degrees up to $n$.
We expect that proving an equivalence between discrete and classical homotopy $n$-types will eventually pave the way to a proof of the aforementioned conjecture.
As a proof of concept, we show that discrete and classical homotopy $1$-types are indeed equivalent.

Our main result is the construction of a fibration category structure \cite{brown,szumilo-cofibration-categories}, a weakening of the definition of a model category \cite{quillen}, on the category of graphs whose weak equivalences are exactly the discrete $n$-equivalences, i.e., maps inducing isomorphisms on the sets of connected components and the first $n$ homotopy groups.
When $n = 1$, we prove that this fibration category is equivalent to the category of groupoids, which is known to model the classical homotopy $1$-types.

The techniques used to prove the result are of independent interest.
Mimicking the construction of a fibration category for discrete homotopy theory \cite{carranza-kapulkin:cubical-graphs}, we wish to pull back this fibration category structure via the cubical nerve functor from a suitable fibration category of cubical sets.
Here, special care is needed to establish a suitable model of homotopy $n$-types on the category of cubical sets and, unfortunately, the naive solution, i.e., a Bousfield localization of the Grothendieck model structure on cubical sets \cite{cisinski-book} would not work.
Indeed, pulling back this fibration category structure would in particular require that we be able to construct a discrete homotopy $n$-type approximating a given graph, which is not currently known.

Instead, we revisit a classic construction of Elvira-Donazar and Hernandez-Paricio \cite{Elvira-Donazar_Hernandez-Paricio_1995} of a model structure for homotopy $n$-types on the category of simplicial sets.
In this model structure, not all objects are cofibrant, which allows for the fibrancy condition to be relaxed.
In particular, all Kan complexes are fibrant.

Our first result, in \cref{sec:simplicial}, is a new construction of this model structure, using the right transfer \cite{hess-kedziorek-riehl-shipley} along the $(n+1)$-coskeleton functor.
Then, in \cref{sec:cubical}, we observe that the proof applies almost verbatim to construct an analogous model structure for homotopy $n$-types on the category of cubical sets.
Finally, in \cref{sec:graphs}, we use the latter to construct a fibration category structure on the category of graphs whose weak equivalences are the $n$-equivalences and prove that, for $n=1$, this fibration category is equivalent to that of groupoids.

\textbf{Acknowledgement.} U.M.~would like to thank David White for pointing him towards \cite[Ex.~2.1.6]{hirschhorn} in answer to a question posted on MathOverflow.

\section{Homotopy \texorpdfstring{$n$}{n}-types of simplicial sets} \label{sec:simplicial}

In this section, we consider two known model category structures for the homotopy $n$-types of simplicial sets, i.e., two model structures where the weak equivalences are the simplicial maps inducing isomorphisms the sets of sets of connected components and the first $n$ simplicial homotopy groups.
The first of these is a left Bousfield localization of the standard Kan--Quillen model structure on the category of simplicial sets, and has been studied, for example, in \cite{hirschhorn} and \cite{cisinski:presheaves}.
The second model structure is due to Elvira-Donazar and Hernandez-Paricio \cite{Elvira-Donazar_Hernandez-Paricio_1995}.
We give a new proof for the existence of this second model structure and show that it is right-transferred from the first model structure along the $(n+1)$-coskeleton functor.
We also prove that the skeleton-coskeleton adjunction gives a Quillen equivalence between these two model structures.

Let $\Delta$ denote the simplex category.
Its objects are the posets $[n] = \left\{0 \leq \cdots \leq n\right\}$ for $n \in \mathbb{N}$, and its morphisms are all order preserving morphisms between these posets, generated under composition by the face maps and the degeneracy maps.
A simplicial set is a presheaf over $\Delta$.
We write $\sSet$ for the presheaf category $\Set^{\Delta^{\op}}$.
The representable presheaf $\Delta(-,[n])$ is denoted $\Delta^{n}$, its boundary is denoted $\partial\Delta^{n}$, and its $i$th horn is denoted $\Lambda^{n}_{i}$.

Consider the set
\[
    J_{n} = \left\{\Lambda^{k}_{i} \hookrightarrow \Delta^{k} \ \mid \ k > 0, \ 0 \leq i \leq k \right\} \cup \left\{\partial\Delta^{k} \hookrightarrow \Delta^{k} \ \mid \ k \geq n+2 \right\}.
\]
A simplicial map $f \from X \to Y$ is \emph{$n$-anodyne} if it belongs to the saturation of $J_{n}$.
A simplicial map $f \from X \to Y$ is a \emph{naive $n$-fibration} if it has the right lifting property with respect to $J_{n}$, and hence with respect to every $n$-anodyne map.
A simplicial set $X$ is \emph{$n$-fibrant} if the unique map $X \to \Delta^{0}$ is a naive $n$-fibration.

A simplicial map $f \from X \to Y$ is an \emph{$n$-equivalence} if it induces a bijection $f_{*} \from \pi_{0}X \to \pi_{0}Y$ and an isomorphism $f_{*} \from \pi_{k}(X,x) \to \pi_{k}(Y,fx)$ for every $0 < k \leq n$ and every $0$-simplex $x \in X_{0}$.
Here, $\pi_{0}X$ denotes the set of path-components of a simplicial set $X$ and $\pi_{k}(X,x)$ denotes the $k$th simplicial homotopy group of a pointed simplicial set $(X,x)$, as defined in \cite[\S~1.7]{goerss-jardine}.

A simplicial map $f \from X \to Y$ is an \emph{$n$-fibration} if it has the right lifting property with respect to every map that is both a monomorphism and an $n$-equivalence.

\begin{theorem}[\protect{\cite[\S~9.2]{cisinski:presheaves}}]\label{thm:cisinski-model-simplicial}
    The category $\sSet$ of simplicial sets admits a cofibrantly generated model category structure, where:
    \begin{itemize}
        \item the cofibrations are the monomorphisms,
        \item the fibrant objects are the $n$-fibrant simplicial sets, and
        \item the weak equivalences are the $n$-equivalences.
    \end{itemize}
    Moreover, a simplicial map $f \from X \to Y$ between $n$-fibrant simplicial sets $X$ and $Y$ is an $n$-fibration if and only if it is a naive $n$-fibration.
    We refer to this model category structure as the \emph{Cisinski model category structure for $n$-types of simplicial sets}, and denote it by $\mathrm{L}_{n}\sSet$.
    It is the left Bousfield localization of the Kan--Quillen model category structure on $\sSet$ at the class of $n$-equivalences.
    \qed
\end{theorem}

We now move to the model category structure constructed in \cite{Elvira-Donazar_Hernandez-Paricio_1995}.
Our treatment differs from the original construction, and we give a new proof for the existence of this model category structure using the transfer theorem of \cite{hess-kedziorek-riehl-shipley}.
We begin by recalling some notation.

\begin{notation}
    Given any integer $n \geq -1$, let $\Delta_{\leq n}$ denote the full subcategory of $\Delta$ on the objects $[0], \ldots, [n]$, and let $\sSet_{\leq n}$ denote the category of presheaves over $\Delta_{\leq n}$.
    The inclusion $i_{n} \from \Delta_{\leq n} \to \Delta$ induces a forgetful functor
    \[
        i_{n}^{*} \from \sSet \to \sSet_{\leq n}.
    \]
    This functor admits both adjoints:
    \[
    \begin{tikzcd}
        \sSet \arrow[rr,"i_{n}^{*}" description] &&  \arrow[ll, bend right,"\bot","{i_{n}}_{*}"']
        \arrow[ll, bend left,"\bot"',"{i_{n}}_{!}"]
        \sSet_{\leq n}
    \end{tikzcd}
    \]
    We write $\sk{n}$ and $\cosk{n}$ for the composite functors
    \[
        \sk{n} =  {i_{n}}_{*} \circ i_{n}^{*} \from \sSet \to \sSet \qquad \text{and} \qquad 
        \cosk{n} =  {i_{n}}_{!} \circ i_{n}^{*} \from \sSet \to \sSet
    \]
    respectively.
    These are called the \emph{$n$-skeleton} and the \emph{$n$-coskeleton} functors respectively.
    These in turn form an adjoint pair:
    \[
    \begin{tikzcd}
        \sSet \arrow[rr,phantom,"\bot" description] \arrow[rr,shift left = 1.5ex,"\sk{n}"] &&
        \arrow[ll, shift left = 1.5ex,"\cosk{n}"]
        \sSet
    \end{tikzcd}
    \]
\end{notation}

\begin{definition}[cf. \protect{\cite[Def.~2.1.3]{hess-kedziorek-riehl-shipley}}]
    \leavevmode
    \begin{enumerate}
        \item A simplicial map $f \from X \to Y$ is a \emph{transferred naive $n$-fibration} if $\cosk{n+1}f \from \cosk{n+1}X \to \cosk{n+1}Y$ is a naive $n$-fibration.
        A simplicial set $X$ is \emph{transferred $n$-fibrant} if the unique map $X \to \Delta^{0}$ is a transferred naive $n$-fibration, or equivalently if $\cosk{n+1}X$ is $n$-fibrant.
        \item A simplicial map $f \from X \to Y$ is a \emph{transferred $n$-fibration} if $\cosk{n+1}f \from \cosk{n+1}X \to \cosk{n+1}Y$ is an $n$-fibration.
        \item A simplicial map $f \from X \to Y$ is a \emph{transferred $n$-cofibration} if it has the left lifting property with respect to every map that is both a transferred $n$-fibration and an $n$-equivalence.
    \end{enumerate}
\end{definition}

We will prove the following theorem.

\begin{theorem}[cf. \protect{\cite[Thm.~2.3]{Elvira-Donazar_Hernandez-Paricio_1995}}]\label{thm:transferred-model-simplicial}
    The category $\sSet$ of simplicial sets admits a cofibrantly generated model category structure, where:
    \begin{itemize}
        \item the cofibrations are the transferred $n$-cofibrations,
        \item the fibrations are the transferred $n$-fibrations, and
        \item the weak equivalences are the $n$-equivalences.
    \end{itemize}
    We refer to this model category structure as the Elvira--Hernandez model category structure for $n$-types of simplicial sets, and denote it by $\sSet_{n}$.
    It is right-transferred from the Cisinski model category structure for $n$-types of simplicial sets along the $(n+1)$-coskeleton functor $\cosk{n+1} \from \sSet \to \sSet$.
    Furthermore, the adjunction
    \[
    \begin{tikzcd}
        \mathrm{L}_{n} \sSet \arrow[rr,phantom,"\bot" description] \arrow[rr,shift left = 1.5ex,"\sk{n+1}"] &&
        \arrow[ll, shift left = 1.5ex,"\cosk{n+1}"]
        \sSet_{n}
    \end{tikzcd}
    \]
    is a Quillen equivalence between the Cisinski and the Elvira--Hernandez models for $n$-types of simplicial sets.
\end{theorem}

At first glance, it seems as though we did not transfer the class of $n$-equivalences.
The following couple of results tells us that the class of $n$-equivalences remains unchanged upon transferring along the $(n+1)$-coskeleton functor.

\begin{lemma}[\protect{\cite[Lem.~2.3]{Elvira-Donazar_Hernandez-Paricio_1995}}]\label{lem:unit-coskeleton-simplicial}
    For every simplicial set $X$, the canonical map $\eta_{X} \from X \to \cosk{n+1}X$ is an $n$-equivalence.
\end{lemma}
\begin{proof}
    This follows from the observation that the map $\eta_{X} \from X \to \cosk{n+1}X$ is a bijection on the $k$-simplices for $0 \leq k \leq n+1$.  
\end{proof}

\begin{proposition}[\protect{\cite[Prop.~2.2]{Elvira-Donazar_Hernandez-Paricio_1995}}]\label{prop:transfer-n-equivalences-simplicial}
    A simplicial map $f \from X \to Y$ is an $n$-equivalence if and only if $\cosk{n+1}f \from \cosk{n+1}X \to \cosk{n+1}Y$ is an $n$-equivalence.
\end{proposition}
\begin{proof}
    Consider the naturality square:
    \[
    \begin{tikzcd}
        X \arrow[r,"\eta_{X}"] \arrow[d,"f"] & \cosk{n+1}X \arrow[d,"\cosk{n+1}f"] \\
        Y \arrow[r,"\eta_{Y}"] & \cosk{n+1}Y
    \end{tikzcd}
    \]
    By \cref{lem:unit-coskeleton-simplicial}, both the horizontal maps in the above diagram are $n$-equivalences.
    Since the class of $n$-equivalences satisfies the 2-out-of-3 property, we have the required result.
\end{proof}

Next, we characterize the transferred \emph{naive} $n$-fibrations in terms of their lifting properties.
We will first need the following lemma.

\begin{lemma}\label{lem:identities-simplicial}
    We have the following identities:
    \[
        \sk{n+1}( \partial\Delta^{k} \hookrightarrow \Delta^{k}) =
        \begin{cases}
            \partial\Delta^{k} \hookrightarrow \Delta^{k} & \text{if } k \leq n+1 \\
            \mathrm{id}_{\sk{n+1}\Delta^{k}} & \text{if } k \geq n+2
        \end{cases}
    \]
    and
    \[
        \sk{n+1}( \Lambda^{k}_{i} \hookrightarrow \Delta^{k}) =
        \begin{cases}
            \Lambda^{k}_{i} \hookrightarrow \Delta^{k} & \text{if } k \leq n+1 \\
            \Lambda^{n+2}_{i} \hookrightarrow \partial\Delta^{n+2} & \text{if } k = n+2 \\
            \mathrm{id}_{\sk{n+1}\Delta^{k}} & \text{if } k \geq n+3.
        \end{cases}
    \]
\end{lemma}
\begin{proof}
    These identities can be verified using skeletal induction for simplicial sets (see, for example, \cite[Thm.~1.3.8]{cisinski-book}).
\end{proof}

\begin{proposition}\label{prop:transfer-naive-n-fibrations-simplicial}
    A simplicial map $f \from X \to Y$ is a transferred naive $n$-fibration if and only if it has the right lifting property with respect to the set
    \[
        J_{n}' = \left\{ \Lambda^{k}_{i} \hookrightarrow \Delta^{k} \ \mid \ 0 < k \leq n+1, \ 0 \leq i \leq k \right\} \cup \left\{ \Lambda^{n+2}_{i} \hookrightarrow \partial\Delta^{n+2} \ \mid \ 0 \leq i \leq n+2  \right\}.
    \]
\end{proposition}
\begin{proof}
    Observe that given two simplicial maps $i \from A \to B$ and $f \from X \to Y$, $f$ has the right lifting property with respect to $\sk{n+1}i$ if and only if $\cosk{n+1}f$ has the right lifting property with respect to $i$.
    The required result then follows from \cref{lem:identities-simplicial}.
\end{proof}

We also have the following characterization of transferred $n$-cofibrations.

\begin{proposition}\label{prop:transfer-n-cofibrations-simplicial}
    A simplicial map $f \from X \to Y$ is both a transferred $n$-fibration and an $n$-equivalence if and only if it has the right lifting property with respect to the set
    \[
        I_{n}' = \left\{ \partial\Delta^{k} \hookrightarrow \Delta^{k} \ \mid \ 0 \leq k \leq n+1 \right\}.
    \]
    Equivalently, a simplicial map $i \from A \to B$ is a transferred $n$-cofibration if and only if it belongs to the saturation of $I_{n}'$.
\end{proposition}
\begin{proof}
    By \cref{prop:transfer-n-equivalences-simplicial}, a map $f \from X \to Y$ is both a transferred $n$-fibration and an $n$-equivalence if and only if
    $\cosk{n+1}f \from \cosk{n+1}X \to \cosk{n+1}Y$ is both an $n$-fibration and an $n$-equivalence.
    By \cref{thm:cisinski-model-simplicial}, we know that $\cosk{n+1}f \from \cosk{n+1}X \to \cosk{n+1}Y$ is both an $n$-fibration and an $n$-equivalence
    if and only if it has the right lifting property with respect to the set
    \[
        I = \left\{ \partial\Delta^{k} \hookrightarrow \Delta^{k} \ \mid \ k \geq 0 \right\}
    \]
    of all boundary inclusions, and hence with respect to every monomorphism.
    Since $\cosk{n+1}f$ has the right lifting property with respect to a simplicial map $i \from A \to B$ if and only if $f$ has the right lifting property with respect to $\sk{n+1}i$, the required result then follows from \cref{lem:identities-simplicial}.
\end{proof}

We will need the following couple of lemmas in order to verify the acyclicity condition in \cite[Cor.~3.3.4]{hess-kedziorek-riehl-shipley}.

\begin{lemma}\label{lem:counit-skeleton-simplicial}
    For every simplicial set $X$, the canonical map $\varepsilon_{X} \from \sk{n+1}X \to X$ is $n$-anodyne.
\end{lemma}
\begin{proof}
    The map $\varepsilon_{X} \from \sk{n+1}X \to X$ is a transfinite composite of the following sequence:
    \[
    \begin{tikzcd}
        \sk{n+1}X \arrow[r] & \sk{n+2}X \arrow[r] & \sk{n+3}X \arrow[r] & \cdots
    \end{tikzcd}
    \]
    and each map in this sequence is $n$-anodyne (see \cite[Thm.~1.3.8]{cisinski-book}).
\end{proof}

\begin{lemma}\label{lem:acyclicity-condition-simplicial}
    Let $i \from A \to B$ be a simplicial map with left lifting property with respect to all transferred $n$-fibrations.
    Then, $i$ is an $n$-equivalence.
\end{lemma}

Before proving \cref{lem:acyclicity-condition-simplicial}, we recall a standard piece of model-categorical notation (cf.~\cite{cisinski-book}).
Given any class $C$ of maps in a category, $\mathscr{L}(C)$ denotes the class of all maps having the left lifting property with respect to every map in $C$ and $\mathscr{R}(C)$ denotes the class of all maps having the right lifting property with respect to every map in $C$.

\begin{proof}
    From the definitions, it follows that a simplicial map $f \from X \to Y$ is a transferred $n$-fibration if and only it has the right lifting property against $\sk{n+1}j$ for every map $j \from C \to D$ that is both a monomorphism and an $n$-equivalence.
    Consider the naturality square:
    \[
    \begin{tikzcd}
        \sk{n+1}C \arrow[r,"\varepsilon_{C}"] \arrow[d,"\sk{n+1}j"'] & C \arrow[d,"j"]  \\
        \sk{n+1}D \arrow[r,"\varepsilon_{D}"] & D
    \end{tikzcd}
    \]
    By \cref{lem:counit-skeleton-simplicial}, both the horizontal maps in the above diagram are $n$-anodyne, and hence, both monomorphisms and $n$-equivalences (see \cite[Prop.~2.4.25]{cisinski-book}).
    Thus, if $j$ is a monomorphism and an $n$-equivalence, so is $\sk{n+1}j$.
    Since the class of maps that are both monomorphisms and $n$-equivalences is saturated (see \cref{thm:cisinski-model-simplicial}), it follows that any map $i \from A \to B$ that has the left lifting property with respect to all transferred $n$-fibrations is both a monomorphism and an $n$-equivalence.
    In other words, we have the following:
    \begin{align*}
        \mathcal{L}{(\text{transferred }n\text{-fibrations})} &= \mathcal{L}(\mathcal{R}(\mathrm{sk}_{n+1}(\text{monos} \cap n\text{-equivalences}))) \\
        &\subseteq \mathcal{L}(\mathcal{R}(\text{monos} \cap n\text{-equivalences})) \\
        &= \text{monos} \cap n\text{-equivalences} \\
        &\subseteq n\text{-equivalences}. \qedhere
    \end{align*}
\end{proof}

\begin{proof}[Proof of \cref{thm:transferred-model-simplicial}]
    \cref{lem:acyclicity-condition-simplicial} allows us to apply \cite[Cor.~3.3.4]{hess-kedziorek-riehl-shipley} to right-transfer the Cisinski model category structure for $n$-types of simplicial sets (see \cref{thm:cisinski-model-simplicial}) along the adjunction $\sk{n+1} \dashv \cosk{n+1}$.
    This proves the existence of the transferred model category structure and also tells us that this adjunction is Quillen.
    By \cref{prop:transfer-n-equivalences-simplicial}, we know that $\cosk{n+1}$ preserves and reflects $n$-equivalences.
    Thus, it suffices to prove that for every simplicial set $X$, the unit $X \to \cosk{n+1}\sk{n+1}X$ of the adjunction $\sk{n+1} \dashv \cosk{n+1}$ is an $n$-equivalence.
    Observe that the functor $i_{n+1}^{*} \circ {i_{n+1}}_{*} \from \sSet_{\leq n+1} \to \sSet_{\leq n+1}$ equals the identity functor on $\sSet_{\leq n+1}$.
    It follows that
    \[
        \cosk{n+1} \circ \sk{n+1} = {i_{n+1}}_{!} \circ i_{n+1}^{*} \circ {i_{n+1}}_{*} \circ i_{n+1}^{*} = {i_{n+1}}_{!} \circ i_{n+1}^{*} = \cosk{n+1}.
    \]
    The required result then follows from \cref{lem:unit-coskeleton-simplicial}.
\end{proof}

\begin{remark}
    \label{remark:transferring-makes-naivete-redundant-simplicial}
    It is perhaps not obvious that the model category structure constructed in \cref{thm:transferred-model-simplicial} coincides with the model category structure of \cite[Thm.~2.3]{Elvira-Donazar_Hernandez-Paricio_1995}.
    Let us now justify this claim.
    The weak equivalences in both model category structures clearly coincide -- these are given in both instances by the $n$-equivalences.
    The cofibrations in both model category structures coincide by virtue of \cref{prop:transfer-n-cofibrations-simplicial} and \cite[Thm.~2.1]{Elvira-Donazar_Hernandez-Paricio_1995} -- these are given in both instances by the maps belonging to the saturation of the set
    \[
        I_{n}' = \left\{ \partial\Delta^{k} \hookrightarrow \Delta^{k} \ \mid \ 0 \leq k \leq n+1 \right\}.
    \]
    This is sufficient to prove that the two model category structures coincide.
    However, while the fibrations in \cref{thm:transferred-model-simplicial} are the transferred $n$-fibrations, the fibrations in \cite[Thm.~2.3]{Elvira-Donazar_Hernandez-Paricio_1995} are the transferred naive $n$-fibrations (see \cref{prop:transfer-naive-n-fibrations-simplicial}).
    Thus, as a consequence of \cite[Thm.~2.3]{Elvira-Donazar_Hernandez-Paricio_1995}, we see that every transferred naive $n$-fibration is a transferred $n$-fibration, and vice-versa.
    This is in contrast with the Cisinski model category structure for $n$-types of simplicial sets, where there exist naive $n$-fibrations that are not $n$-fibrations (see \cite[Ex.~2.1.6]{hirschhorn}).
\end{remark}

\section{Homotopy \texorpdfstring{$n$}{n}-types of cubical sets} \label{sec:cubical}

\subsection{Background on cubical sets}

We very briefly recall some background on cubical sets, the Grothendieck model category structure, and cubical homotopy groups.
For a more detailed introduction, we direct the reader to \cite{carranza-kapulkin:homotopy-groups}.
Other references on this topic include \cite{cisinski:presheaves}, \cite{cisinski2014univalent}, \cite{jardine:categorical-homotopy}, \cite{kapulkin-voevodsky}, and \cite{doherty-kapulkin-lindsey-sattler}.

Let $\boxcat$ denote the \emph{cube category}.
Its objects are the posets $[1]^{n} = \left\{0 \leq 1\right\}^{n}$ for $n \in \mathbb{N}$, and its morphisms are generated (inside the category of posets) under composition by the face maps, the degeneracy maps, and connections of both kinds (positive and negative).
A \emph{cubical set} is a presheaf over $\boxcat$.
We write $\cSet$ for the presheaf category $\Set^{\boxcat^{\op}}$.
The \emph{combinatorial $n$-cube} $\cube{n}$ is given by the representable presheaf $\boxcat(-,[1]^{n})$, its \emph{boundary} is denoted $\bdcube{n}$, and its \emph{$(i,\epsilon)$-open box} is denoted $\obox{n}{i,\epsilon}$.

Consider the set of all open box inclusions:
\[
    J = \left\{ \obox{n}{i,\epsilon} \hookrightarrow \cube{n} \ \mid \ n > 0, \ 1 \leq i \leq n, \ \epsilon = 0,1 \right\}.
\]
A cubical map $f \from X \to Y$ is \emph{anodyne} if it belongs to the saturation of $J$.
A cubical map $f \from X \to Y$ is a \emph{Kan fibration} if it has the right lifting property with respect to $J$, and hence with respect to every anodyne map.
A cubical set $X$ is a \emph{Kan complex} if the unique map $X \to \cube{0}$ is a Kan fibration.
We write $\mathsf{Kan}$ for the full subcategory of $\cSet$ on Kan complexes.

Let $\otimes \from \boxcat \times \boxcat \to \boxcat$ denote the functor given by the assignment $([1]^{m},[1]^{n}) \mapsto [1]^{m+n}$.
Post-composing it with the Yoneda embedding $\boxcat \hookrightarrow \cSet$ and taking its left Kan extension along $\boxcat \times \boxcat \hookrightarrow \cSet \times \cSet$, we obtain the \emph{geometric product} of cubical sets
\[
\begin{tikzcd}
    \boxcat \times \boxcat \arrow[r,"\otimes"] \arrow[d,hook] & \boxcat \arrow[r,hook] & \cSet \\
    \cSet \times \cSet \arrow[urr,"\otimes"'] &&
\end{tikzcd}
\]
This defines a monoidal structure on $\cSet$.

Given two cubical maps $f, g \from X \to Y$, an \emph{elementary homotopy} from $f$ to $g$ is a map $H \from X \otimes \cube{1} \to Y$ such that the following diagram commutes:
\[
\begin{tikzcd}
    X \otimes \cube{0} \arrow[d,"\partial_{1,0}"'] \arrow[dr,"f"] & \\
    X \otimes \cube{1} \arrow[r,"H"] & Y \\
    X \otimes \cube{0} \arrow[u,"\partial_{1,1}"] \arrow[ur,"g"'] 
\end{tikzcd}
\]
Given two cubical sets $X$ and $Y$, let $[X,Y]$ denote the quotient of the hom-set $\cSet(X,Y)$ by the equivalence relation generated by elementary homotopies.
Two cubical maps $f, g \from X \to Y$ are \emph{homotopic} if they have the same equivalence class in $[X,Y]$.
A cubical map $f \from X \to Y$ is a \emph{homotopy equivalence} if there exists a cubical map $g \from Y \to X$ such that $gf$ is homotopic to $\text{id}_{X}$ and $fg$ is homotopic to $\text{id}_{Y}$.

A cubical map $f \from X \to Y$ is a \emph{weak homotopy equivalence} if, for every Kan complex $Z$, the induced map
\[
    f^{*} \from [Y,Z] \to [X,Z]
\]
is a bijection.

\begin{theorem}[Cisinski, see \protect{\cite[Thm.~1.34]{doherty-kapulkin-lindsey-sattler}}]\label{thm:grothendieck}
    The category $\cSet$ admits a cofibrantly generated model category structure where
    \begin{itemize}
        \item the cofibrations are the monomorphisms,
        \item the fibrations are the Kan fibrations, and
        \item the weak equivalences are the weak homotopy equivalences.
    \end{itemize}
    We refer to this model category structure as the Grothendieck model category structure.
    \qed
\end{theorem}

Let $T \from \boxcat \to \sSet$ denote the functor given by the assignment $[1]^{n} \mapsto (\simp{1})^{n}$.
Taking its left Kan extension along the Yoneda embedding $\boxcat \hookrightarrow \cSet$, we obtain the \emph{triangulation functor} $T \from \cSet \to \sSet$ and its right adjoint $U \from \sSet \to \cSet$ given by $(UX)_{n} = \sSet((\simp{1})^{n},X)$.
\[
\begin{tikzcd}[row sep = large, column sep = large]
    \boxcat \arrow[r,"{[1]^{n} \mapsto (\simp{1})^{n}}"] \arrow[d,hook] & \sSet \arrow[dl,"U",shift left = 2ex] \\
    \cSet \arrow[ur,"T"] &
\end{tikzcd}
\]

\begin{theorem}[\protect{\cite[Thm.~6.26]{doherty-kapulkin-lindsey-sattler}}]\label{thm:triangulation}
    The adjunction $T \dashv U$ is a Quillen equivalence between the Grothendieck model category structure on $\cSet$ and the Quillen model category structure on $\sSet$.
    \qed
\end{theorem}

A \emph{relative cubical set} $(X,A)$ is a cubical set $X$ equipped with a cubical subset $A \subseteq X$.
A \emph{relative cubical map} $f \from (X,A) \to (Y,B)$ is a cubical map $f \from X \to Y$ that satisfies $f(A) \subseteq B$.
Given two relative cubical maps $f, g \from (X,A) \to (Y,B)$, a \emph{relative elementary homotopy} from $f$ to $g$ is an elementary homotopy $H \from X \otimes \cube{1} \to Y$ from $f$ to $g$ that satisfies $H(A \otimes \cube{1}) \subseteq B$.
Given two relative cubical sets $(X,A)$ and $(Y,B)$, let $[(X,A),(Y,B)]$ denote the quotient of the set of all relative cubical maps $(X,A) \to (Y,B)$ by the equivalence relation generated by relative elementary homotopies.

The \emph{set of path components} $\pi_{0}X$ of a Kan complex $X$ is given by
\[
    \pi_{0}X = [\cube{0},X].
\]
For $n > 0$, the \emph{$n$th homotopy group} $\pi_{n}(X,x)$ of a pointed Kan complex $(X,x)$ is given by
\[
    \pi_{n}(X,x) = [(\cube{n},\bdcube{n}),(X,x)],
\]
or equivalently,
\[
    \pi_{n}(X,x) = [(\bdcube{n+1},0),(X,x)].
\]

The \emph{set of path components} $\pi_{0}X$ of a cubical set $X$ is given by the set of path components $\pi_{0}\tilde{X}$ of its fibrant replacement in the Grothendieck model category structure on $\cSet$.
Similarly, for $n > 0$, the \emph{$n$th homotopy group} $\pi_{n}(X,x)$ of a pointed cubical set $(X,x)$ is given by the $n$th homotopy group $\pi_{n}(\tilde{X},x)$ of its fibrant replacement in the Grothendieck model category structure on $\cSet$.
(Note that $x \in X$ can be regarded as an element of $\tilde{X}$ via the inclusion $X \hookrightarrow \tilde{X}$ coming from fibrant replacement.)

We proceed to show that the homotopy groups of a cubical set $X$ agree with the simplicial homotopy groups of its triangulation $TX$.

Let $\lvert - \rvert_{\boxcat} \from \boxcat \to \Top$ denote the functor given by the assignment $[1]^{n} \mapsto [0,1]^{n}$.
Taking its left Kan extension along the Yoneda embedding $\boxcat \hookrightarrow \cSet$, we obtain the \emph{geometric realization functor} $\lvert - \rvert_{\boxcat} \from \cSet \to \Top$ and its right adjoint $\mathrm{Sing}_{\boxcat} \from \Top \to \cSet$ given by $\mathrm{Sing}_{\boxcat}(X)_{n} = \Top([0,1]^{n},X)$.
\[
\begin{tikzcd}[row sep = large, column sep = large]
    \boxcat \arrow[r,"{[1]^{n} \mapsto [0,1]^{n}}"] \arrow[d,hook] & \Top \arrow[dl,"\mathrm{Sing}_{\boxcat}",shift left = 2ex] \\
    \cSet \arrow[ur,"\lvert - \rvert_{\boxcat}"] &
\end{tikzcd}
\]

\begin{proposition}\label{prop:comparison-homotopy-groups}
    Given any cubical set $X$, we have a bijection $\pi_{0}X \cong \pi_{0}TX$ and a group isomorphism $\pi_{n}(X,x) \cong \pi_{n}(TX,x)$ for every $n >0$ and every $0$-cube $x \in X_{0}$.
\end{proposition}
\begin{proof}
    Let $\tilde{X}$ be a fibrant replacement of $X$ in the Grothendieck model category structure on $\cSet$, and let $\tilde{T\tilde{X}}$ be a fibrant replacement of $T\tilde{X}$ in the Quillen model category structure on $\sSet$.
    Then, we have:
    \[
    \begin{array}{r l l}
        \pi_{n}(X,x)
        &= \pi_{n}(\tilde{X},x) &
        \\
        &\cong \pi_{n}(\lvert \tilde{X} \rvert_{\boxcat},x)
        & \text{by \cite[Thm.~3.25]{carranza-kapulkin:homotopy-groups}} \\
        &\cong \pi_{n}(\lvert T\tilde{X} \rvert_{\Simp},x)
        & \text{by \cite[Lem.~2.23]{carranza-kapulkin:homotopy-groups}} \\
        &\cong \pi_{n}(\lvert \tilde{T\tilde{X}} \rvert_{\Simp},x) 
        & \text{by \cite[Prop.~2.3.5]{quillen}} \\
        &\cong \pi_{n}(\tilde{T\tilde{X}},x)
        & \text{by \cite[Prop.~3.6.3]{hovey}} \\
        &= \pi_{n}(T\tilde{X},x)
        \\
        &\cong \pi_{n}(TX,x)
        & \text{by \cite[Thm.~6.26]{doherty-kapulkin-lindsey-sattler}.}
    \end{array}
    \]
    where $\lvert - \rvert_{\Simp} \from \sSet \to \Top$ denotes the geometric realization of simplicial sets.
\end{proof}

Next, we state Whitehead's theorem for cubical sets, which gives a characterization of the weak equivalences in the Grothendieck model category structure on $\cSet$ in terms of the homotopy groups.

\begin{theorem}[\protect{\cite[Thm.~4.7]{carranza-kapulkin:homotopy-groups}}]\label{thm:whitehead}
    A cubical map $f \from X \to Y$ between Kan complexes is a homotopy equivalence if and only if it induces a bijection $f_{*} \from \pi_{0}X \to \pi_{0}Y$ and an isomorphism $f_{*} \from \pi_{n}(X,x) \to \pi_{n}(Y,fx)$ for every $n > 0$ and every $0$-cube $x \in X_{0}$.
    \qed
\end{theorem}

\begin{corollary}\label{cor:whitehead}
    A cubical map $f \from X \to Y$ is a weak homotopy equivalence if and only if it induces a bijection $f_{*} \from \pi_{0}X \to \pi_{0}Y$ and an isomorphism $f_{*} \from \pi_{n}(X,x) \to \pi_{n}(Y,fx)$ for every $n > 0$ and every $0$-cube $x \in X_{0}$.
\end{corollary}
\begin{proof}
    Let $\tilde{f} \from \tilde{X} \to \tilde{Y}$ be a functorial fibrant replacement of $f \from X \to Y$ in the Grothendieck model category structure on $\cSet$.
    Then, $f$ is a weak homotopy equivalence if and only if $\tilde{f}$ is.
    But a cubical map between Kan complexes is a weak homotopy equivalence if and only if it is a homotopy equivalence (see \cite[Thm.~2.4.26]{cisinski-book}).
    The required result then follows from \cref{thm:whitehead}.
\end{proof}

\subsection{The Cisinski model category structure for $n$-types of cubical sets}

We use Cisinski theory to construct a model category structure on $\cSet$, analogous to the model category structure of \cref{thm:cisinski-model-simplicial}.

Fix an integer $n \geq 0$.

\begin{definition}
    \leavevmode
    \begin{enumerate}
        \item A cubical map $f \from X \to Y$ is \emph{$n$-anodyne} if it belongs to the saturation of the set
        \[
            J_{n} = \left\{ \obox{k}{i,\varepsilon} \hookrightarrow \cube{k} \ \mid \ k > 0, \ 1 \leq i \leq k, \ \varepsilon = 0, 1 \right\} \cup \left\{\bdcube{k} \hookrightarrow \cube{k} \ \mid \ k \geq n+2 \right\}.
        \]
        Note that, following the notation of \cite[Ex.~2.4.13]{cisinski-book}, the class of $n$-anodyne maps is precisely the class $\mathrm{An}_{(-) \otimes \cube{1}}(\left\{\bdcube{n+2} \hookrightarrow \cube{n+2}\right\},\left\{\bdcube{k} \hookrightarrow \cube{k} \ \mid \ k \geq 0\right\})$.
        
        \item A cubical map $f \from X \to Y$ is a \emph{naive $n$-fibration} if it has the right lifting property with respect to every $n$-anodyne map.
        A cubical set $X$ is \emph{$n$-fibrant} if the unique map $X \to \cube{0}$ is a naive $n$-fibration.
        \item A cubical map $f \from X \to Y$ is an \emph{$n$-equivalence} if, for every $n$-fibrant cubical set $Z$, the induced map
        \[
            f^{*} \from [Y,Z] \to [X,Z]
        \]
        is a bijection.
        \item A cubical map $f \from X \to Y$ is an \emph{$n$-fibration} if it has the right lifting property with respect to every map that is both a monomorphism and an $n$-equivalence.
    \end{enumerate}
\end{definition}

Note that every naive $n$-fibration is, in particular, a Kan fibration.
Thus, every $n$-fibrant cubical set is a Kan complex.
It follows that every weak homotopy equivalence is, in particular, an $n$-equivalence.

\begin{theorem}\label{thm:cisinski-model-cubical}
    The category $\cSet$ admits a cofibrantly generated model category structure where:
    \begin{itemize}
        \item the cofibrations are the monomorphisms,
        \item the fibrant objects are the $n$-fibrant cubical sets, and
        \item the weak equivalences are the $n$-equivalences.
    \end{itemize}
    Moreover, a cubical map $f \from X \to Y$ between $n$-fibrant cubical sets $X$ and $Y$ is an $n$-fibration if and only if it is a naive $n$-fibration.
    We refer to this model category structure as the Cisinski model category structure for $n$-types of cubical sets, and denote it by $\mathrm{L}_{n} \cSet$.
    It is the left Bousfield localization of the Grothendieck model category structure on $\cSet$ at the class of $n$-equivalences.
\end{theorem}
\begin{proof}
    This follows from a straightforward application of \cite[Thm.~2.4.19]{cisinski-book} with the cylinder functor $(-) \otimes \cube{1}$ and the class of anodyne extensions given by the $n$-anodyne maps, i.e.,
    \[
        \mathrm{An}_{(-) \otimes \cube{1}}(\left\{\bdcube{n+2} \hookrightarrow \cube{n+2}\right\},\left\{\bdcube{k} \hookrightarrow \cube{k} \ \mid \ k \geq 0\right\}). \qedhere
    \]
\end{proof}

The following result gives an alternative characterization of $n$-fibrant cubical sets.

\begin{proposition}\label{prop:n-fibrant-characterization}
    A cubical set $X$ is $n$-fibrant if and only if it is a Kan complex with $\pi_{k}(X,x) \cong 0$ for all $k \geq n+1$ and all $0$-cubes $x \in X_{0}$.
\end{proposition}
\begin{proof}
    Recall that the $k$th homotopy group $\pi_{k}(X,x)$ of a pointed Kan complex $(X,x)$ is given by the set $[(\bdcube{k+1},0),(X,x)]$.
    Thus, for any Kan complex $X$, the map $X \to \cube{0}$ has the right lifting property with respect to the boundary inclusion $\bdcube{k+1} \hookrightarrow \cube{k+1}$ if and only if $\pi_{k}(X,x) \cong 0$ for every $0$-cube $x \in X$.
\end{proof}

Next, we give a characterization of $n$-equivalences in terms of homotopy groups, analogous to \cref{cor:whitehead}.
We begin with a technical lemma.

\begin{lemma}\label{lem:n-bijective-maps}
    Let $f \from X \to Y$ be a cubical map that is a bijection on the $k$-cubes for each $0 \leq k \leq n+1$.
    Then, the induced map $f_{*} \from \pi_{0}X \to \pi_{0}Y$ is a bijection, and the induced group homomorphism $f_{*} \from \pi_{k}(X,x) \to \pi_{k}(Y,fx)$ is an isomorphism for each $0 < k \leq n$ and each $0$-cube $x \in X_{0}$.
\end{lemma}
\begin{proof}
    We begin by fixing some notation.
    Consider the following commuting diagram:
    \[
    \begin{tikzcd}
        X_{0} \arrow[r] \arrow[d,"f_{0}"'] & X_{1} \arrow[r] \arrow[d,"f_{1}"'] & X_{2} \arrow[r] \arrow[d,"f_{2}"'] & \cdots \\
        Y_{0} \arrow[r] & Y_{1} \arrow[r] & Y_{2} \arrow[r] & \cdots
    \end{tikzcd}
    \]
    where $X_{0} = X$, $Y_{0} = Y$, $f_{0} = f$, and the maps $X_{j} \to X_{j+1}$ and $Y_{j} \to Y_{j+1}$ are defined by the following pushouts:
    \[
    \begin{tikzcd}
        \displaystyle\coprod_{s \in S_{j}}{\obox{k_{s}}{i_{s},\varepsilon_{s}}} \arrow[d,hook] \arrow[r] \arrow[dr,phantom,"\ulcorner" description, very near end] & X_{j} \arrow[d]  \\
        \displaystyle\coprod_{s \in S_{j}}{\cube{k_{s}}} \arrow[r] & X_{j+1}
    \end{tikzcd}
    \qquad \text{and} \qquad
    \begin{tikzcd}
        \displaystyle\coprod_{t \in T_{j}}{\obox{k_{t}}{i_{t},\varepsilon_{t}}} \arrow[d,hook] \arrow[r] \arrow[dr,phantom,"\ulcorner" description, very near end] & Y_{j} \arrow[d]  \\
        \displaystyle\coprod_{t \in T_{j}}{\cube{k_{t}}} \arrow[r] & Y_{j+1}
    \end{tikzcd}
    \]
    where
    \[
        S_{j} = \{ \obox{k}{i,\varepsilon} \to X_{j} \ \mid \ k > 0, \ 1 \leq i \leq k, \ \varepsilon = 0,1 \}
    \]
    and
    \[
        T_{j} = \{ \obox{k}{i,\varepsilon} \to Y_{j} \ \mid \ k > 0, \ 1 \leq i \leq k, \ \varepsilon = 0,1 \}.
    \]
    Thus, $\tilde{X} := \colim X_{j}$ and $\tilde{Y} := \colim Y_{j}$ are the functorial fibrant replacements of $X$ and $Y$ respectively, in the Grothendieck model category structure on $\cSet$ (see, for example, \cite[Thm.~2.1.14]{hovey}).

    We will show that each of the maps $f_{j} \from X_{j} \to Y_{j}$ is a bijection on cubes of dimensions $\leq n+1$.
    From this, it will follow that the induced map $\tilde{f} \from \tilde{X} \to \tilde{Y}$ is also a bijection on the cubes of dimensions $\leq n+1$.
    Since $\tilde{X}$ and $\tilde{Y}$ are Kan complexes, we have:
    \[
        \pi_{k}(\tilde{X},x) = [(\cube{k},\bdcube{k}),(X,x)]
        \qquad \text{and} \qquad
        \pi_{k}(\tilde{Y},fx) = [(\cube{k},\bdcube{k}),(X,fx)]
    \]
    for every $k > 0$ and $0$-cube $x \in X_{0}$.
    Thus, the map $\tilde{f}$ induces an isomorphism on the homotopy groups in degrees $\leq n$.
    
    We proceed by induction on $j$.
    The base case ($j=0$) is true by assumption.
    For the induction step, we suppose that $f_{j} \from X_{j} \to Y_{j}$ is a bijection on cubes of dimensions $\leq n+1$.

    Consider the sets
    \[
        S_{j,\leq n+1} = \{ \obox{k}{i,\varepsilon} \to X_{j} \ \mid \ 0 < k \leq n+1, \ 1 \leq i \leq k, \ \varepsilon = 0,1 \},
    \]
    \[
        S_{j,\geq n+2} = \{ \obox{k}{i,\varepsilon} \to X_{j} \ \mid \ k \geq n+2, \ 1 \leq i \leq k, \ \varepsilon = 0,1 \},
    \]
    \[
        T_{j,\leq n+1} = \{ \obox{k}{i,\varepsilon} \to Y_{j} \ \mid \ 0 < k \leq n+1, \ 1 \leq i \leq k, \ \varepsilon = 0,1 \}, \text{ and }
    \]
    \[
        T_{j,\geq n+2} = \{ \obox{k}{i,\varepsilon} \to Y_{j} \ \mid \ k \geq n+2, \ 1 \leq i \leq k, \ \varepsilon = 0,1 \}.
    \]
    Then, we have 
    \[
        S_{j} = S_{j,\leq n+1} \sqcup S_{j,\geq n+2}
        \qquad \text{and} \qquad
        T_{j} = T_{j,\leq n+1} \sqcup T_{j,\geq n+2}.
    \]
    Furthermore, since $f_{j} \from X_{j} \to Y_{j}$ is a bijection on cubes of dimensions $\leq n+1$, we also have a bijection
    \[
        S_{j,\leq n+1} \cong T_{j, \leq n+1}.
    \]
    Thus, any $k$-cubes that are glued to $X_{j}$ to obtain $X_{j+1}$ are also glued to $Y_{j}$ to obtain $Y_{j+1}$ (and vice-versa), for all $k \leq n+1$.
    Thus, the induced map $f_{j+1} \from X_{j+1} \to Y_{j+1}$ is also a bijection on cubes of dimensions $\leq n+1$, completing the induction step.
\end{proof}

\begin{theorem}\label{thm:n-equivalences-characterization}
    A cubical map $f \from X \to Y$ is an $n$-equivalence if and only if it induces a bijection $f_{*} \from \pi_{0}X \to \pi_{0}Y$ and an isomorphism $f_{*} \from \pi_{k}(X,x) \to \pi_{k}(Y,fx)$ for every $0 < k \leq n$ and every $0$-cube $x \in X_{0}$.
\end{theorem}
\begin{proof}
    We begin by fixing some notation.
    Consider the following sequence:
    \[
    \begin{tikzcd}
        X_{0} \arrow[r,"g_{0}"] & X_{1} \arrow[r,"g_{1}"] & X_{2} \arrow[r,"g_{2}"] & \cdots
    \end{tikzcd}
    \]
    where $X_{0} = X$ and the maps $g_{j} \from X_{j} \to X_{j+1}$ are defined by the following pushouts:
    \[
    \begin{tikzcd}
        \displaystyle\coprod_{s \in S_{j}}{\obox{k_{s}}{i_{s},\varepsilon_{s}}} \arrow[d,hook] \arrow[r] \arrow[dr,phantom,"\ulcorner" description, very near end] & X_{j} \arrow[d,"g_{j}"]  \\
        \displaystyle\coprod_{s \in S_{j}}{\cube{k_{s}}} \arrow[r] & X_{j+1}
    \end{tikzcd}
    \]
    where
    \[
    S_{j} = \{ \obox{k}{i,\varepsilon} \to X_{j} \ \mid \ k > 0, \ 1 \leq i \leq k, \ \varepsilon = 0,1 \}
    \]
    if $j$ is even, and
    \[
    \begin{tikzcd}
        \displaystyle\coprod_{t \in T_{j}}{\bdcube{k_{t}}} \arrow[d,hook] \arrow[r] \arrow[dr,phantom,"\ulcorner" description, very near end] & X_{j} \arrow[d,"g_{j}"]  \\
        \displaystyle\coprod_{t \in T_{j}}{\cube{k_{t}}} \arrow[r] & X_{j+1}
    \end{tikzcd}
    \]
    where
    \[
        T_{j} = \{ \bdcube{k} \to X_{j} \ \mid \ k \geq n+2 \}
    \]
    if $j$ is odd.
    
    Let $\tilde{X} := \colim X_{j}$ and let $g \from X \to \tilde{X}$ be the induced map.
    Then, $\tilde{X}$ is $n$-fibrant and the map $g \from X \to \tilde{X}$ is $n$-anodyne.
    Furthermore, each $g_{j} \from X_{j} \to X_{j+1}$ induces an isomorphism on the homotopy groups in degrees $\leq n$.
    Indeed, if $j$ is even, then $g_{j}$ is anodyne and in particular, a weak homotopy equivalence.
    Thus, in this case, it induces an isomorphism on all homotopy groups by \cref{cor:whitehead}.
    If $j$ is odd, then $g_{j}$ is a bijection on the $k$-cubes for each $0 \leq k \leq n+1$.
    Thus, in this case, it induces an isomorphisms on the homotopy groups in degrees $\leq n$ by \cref{lem:n-bijective-maps}.
    It follows that $g \from X \to \tilde{X}$ also induces isomorphisms on the homotopy groups in degrees $\leq n$.

    Applying the same construction to $Y$ instead of $X$, we obtain an $n$-fibrant cubical set $\tilde{Y}$ along with a map $h \from Y \to \tilde{Y}$ that is $n$-anodyne and that induces isomorphisms on the homotopy groups in degrees $\leq n$.
    We also obtain an induced map $\tilde{f} \from \tilde{X} \to \tilde{Y}$ that makes the following square commute:
    \[
    \begin{tikzcd}
        X \arrow[d,"f"'] \arrow[r,"g"] & \tilde{X} \arrow[d,"\tilde{f}"] \\
        Y \arrow[r,"h"] & \tilde{Y}
    \end{tikzcd}
    \]
    
    Since $g$ and $h$ are both $n$-anodyne maps and, in particular, $n$-equivalences (see \cite[Prop.~2.4.25]{cisinski-book}), and since the class of $n$-equivalences satisfies the 2-out-of-3 property, $f$ is an $n$-equivalence if and only if $\tilde{f}$ is an $n$-equivalence.
    Since $\tilde{X}$ and $\tilde{Y}$ are both $n$-fibrant, $\tilde{f}$ is an $n$-equivalence if and only if it is a homotopy equivalence (see \cite[Thm.~2.4.26]{cisinski-book}).
    By \cref{thm:whitehead}, $\tilde{f}$ is a homotopy equivalence if and only if it induces isomorphisms on all homotopy groups.
    Since $g \from X \to \tilde{X}$ and $h \from Y \to \tilde{Y}$ both induce isomorphisms on all homotopy groups in degrees $\leq n$, and since $\tilde{X}$ and $\tilde{Y}$ have trivial homotopy groups in degrees above $n$ (see \cref{prop:n-fibrant-characterization}), $\tilde{f}$ induces isomorphisms on all homotopy groups if and only if $f$ induces isomorphisms on all homotopy groups in degrees $\leq n$.
\end{proof}

The above characterization of $n$-equivalences in $\cSet$ allows us to compare them to the $n$-equivalences in $\sSet$.

\begin{proposition}\label{cor:triangulation-n-equivalences}
    A cubical map $f \from X \to Y$ is an $n$-equivalence in $\cSet$ if and only if its triangulation $Tf \from TX \to TY$ is an $n$-equivalence in $\sSet$.
\end{proposition}
\begin{proof}
    This follows from \cref{prop:comparison-homotopy-groups} and \cref{thm:n-equivalences-characterization}.
\end{proof}

\begin{theorem}\label{thm:cisinski-models-are-equivalent}
    The adjunction
    \[
    \begin{tikzcd}
        \mathrm{L}_{n}\cSet
        \arrow[rr,phantom,"\bot" description]
        \arrow[rr,shift left = 1.5ex,"T"] &&
        \arrow[ll, shift left = 1.5ex,"U"]
        \mathrm{L}_{n}\sSet
    \end{tikzcd}
    \]
    is a Quillen equivalence between the Cisinski model category structures for $n$-types of cubical sets and of simplicial sets.
\end{theorem}
\begin{proof}
    We know that the adjunction $T \dashv U$ is a Quillen equivalence between the Grothendieck model category structure on $\cSet$ and the Quillen model category structure on $\sSet$ 
    (see \cref{thm:triangulation}).
    The Cisinski model category structure for $n$-types of cubical sets is the left Bousfield localization of the Grothendieck model category structure on $\cSet$ at the class of $n$-equivalences in $\cSet$, and the Cisinski model category structure for $n$-types of simplicial sets is the left Bousfield localization of the Kan--Quillen model category structure on $\sSet$ at the class of $n$-equivalences in $\sSet$.
    The required result then follows from \cref{cor:triangulation-n-equivalences} and \cite[Thm.~3.3.20]{hirschhorn}.
\end{proof}

\subsection{The transferred model category structure for $n$-types of cubical sets}

We construct a model category structure on $\cSet$, analogous to the model category structure in \cref{thm:transferred-model-simplicial}.

\begin{notation}
    Given any integer $n \geq -1$, let $\boxcat_{\leq n}$ denote the full subcategory of $\boxcat$ on the objects $[1]^{0}, \ldots, [1]^{n}$, and let $\cSet[\leq n]$ denote the category of presheaves over $\boxcat_{\leq n}$.
    The inclusion $i_{n} \from \boxcat_{\leq n} \to \boxcat$ induces a forgetful functor
    \[
        i_{n}^{*} \from \cSet \to \cSet[\leq n].
    \]
    This functor admits both adjoints:
    \[
    \begin{tikzcd}
        \cSet \arrow[rr,"i_{n}^{*}" description] &&  \arrow[ll, bend right,"\bot","{i_{n}}_{*}"']
        \arrow[ll, bend left,"\bot"',"{i_{n}}_{!}"]
        \cSet[\leq n]
    \end{tikzcd}
    \]
    We write $\sk{n}$ and $\cosk{n}$ for the composite functors
    \[
        \sk{n} =  {i_{n}}_{*} \circ i_{n}^{*} \from \cSet \to \cSet \qquad \text{and} \qquad 
        \cosk{n} =  {i_{n}}_{!} \circ i_{n}^{*} \from \cSet \to \cSet
    \]
    respectively.
    These are called the \emph{$n$-skeleton} and the \emph{$n$-coskeleton} functors respectively.
    These in turn form an adjoint pair:
    \[
    \begin{tikzcd}
        \cSet \arrow[rr,phantom,"\bot" description] \arrow[rr,shift left = 1.5ex,"\sk{n}"] &&
        \arrow[ll, shift left = 1.5ex,"\cosk{n}"]
        \cSet.
    \end{tikzcd}
    \]
\end{notation}

Fix an integer $n \geq 0$.
We want to right-transfer the model category structure constructed in \cref{thm:cisinski-model-cubical} along the $(n+1)$-coskeleton functor $\cosk{n+1} \from \cSet \to \cSet$.
Thus, we will be dealing with the following classes of maps:

\begin{definition}[cf.~\protect{\cite[Def.~2.1.3]{hess-kedziorek-riehl-shipley}}]
\label{def:transferred-model}
    \leavevmode
    \begin{enumerate}
        \item A cubical map $f \from X \to Y$ is a \emph{transferred naive $n$-fibration} if $\cosk{n+1}f \from \cosk{n+1}X \to \cosk{n+1}Y$ is a naive $n$-fibration.
        A cubical set $X$ is \emph{transferred $n$-fibrant} if the unique map $X \to \cube{0}$ is a transferred naive $n$-fibration, or equivalently if $\cosk{n+1}X$ is $n$-fibrant.
        \item A cubical map $f \from X \to Y$ is a \emph{transferred $n$-fibration} if $\cosk{n+1}f \from \cosk{n+1}X \to \cosk{n+1}Y$ is an $n$-fibration.
        \item A cubical map $f \from X \to Y$ is a \emph{transferred $n$-cofibration} if it has the left lifting property with respect to every map that is both a transferred $n$-fibration and an $n$-equivalence.
    \end{enumerate}
\end{definition}

We will prove the following theorem:

\begin{theorem}[cf.~\cref{thm:transferred-model-simplicial}]\label{thm:transferred-model-cubical}
    The category $\cSet$ admits a cofibrantly generated model category structure where:
    \begin{itemize}
        \item the cofibrations are the transferred $n$-cofibrations,
        \item the fibrations are the transferred $n$-fibrations, and
        \item the weak equivalences are the $n$-equivalences.
    \end{itemize}
    We refer to this model category structure as the transferred model category structure for $n$-types of cubical sets, and denote it by $\cSet[n]$.
    It is right-transferred from the Cisinski model category structure for $n$-types of cubical sets along the $(n+1)$-coskeleton functor $\cosk{n+1} \from \cSet \to \cSet$.
    Furthermore, the adjunction
    \[
    \begin{tikzcd}
        \mathrm{L}_{n} \cSet \arrow[rr,phantom,"\bot" description] \arrow[rr,shift left = 1.5ex,"\sk{n+1}"] &&
        \arrow[ll, shift left = 1.5ex,"\cosk{n+1}"]
        \cSet[n]
    \end{tikzcd}
    \]
    is a Quillen equivalence between the Cisinski and the transferred models for $n$-types of cubical sets.
\end{theorem}

At first glance, it seems as though we did not transfer the class of $n$-equivalences.
The following couple of results tell us that the class of $n$-equivalences remains unchanged upon transferring along the $(n+1)$-coskeleton functor.

\begin{lemma}[cf.~\cref{lem:unit-coskeleton-simplicial}]\label{lem:unit-coskeleton}
    For every cubical set $X$, the canonical map $\eta_{X} \from X \to \cosk{n+1}X$ is an $n$-equivalence.
\end{lemma}
\begin{proof}
    The map $\eta_{X} \from X \to \cosk{n+1}X$ is a bijection on the $k$-cubes for each $0 \leq k \leq n+1$.
    The required result then follows from \cref{lem:n-bijective-maps} and \cref{thm:n-equivalences-characterization}.
\end{proof}

\begin{proposition}[cf.~\cref{prop:transfer-n-equivalences-simplicial}]\label{prop:transfer-n-equivalences}
    A cubical map $f \from X \to Y$ is an $n$-equivalence if and only if $\cosk{n+1}f \from \cosk{n+1}X \to \cosk{n+1}Y$ is an $n$-equivalence.
\end{proposition}
\begin{proof}
    Consider the naturality square:
    \[
    \begin{tikzcd}
        X \arrow[r,"\eta_{X}"] \arrow[d,"f"'] & \cosk{n+1}X \arrow[d,"\cosk{n+1}f"]  \\
        Y \arrow[r,"\eta_{Y}"] & \cosk{n+1}Y
    \end{tikzcd}
    \]
    By \cref{lem:unit-coskeleton}, both the horizontal maps in the above diagram are $n$-equivalences.
    Since the class of $n$-equivalences satisfies the 2-out-of-3 property, we have the required result.
\end{proof}

Next, we characterize the transferred \emph{naive} $n$-fibrations in terms of their lifting properties.
We will first need the following lemma.

\begin{lemma}[cf. \cref{lem:identities-simplicial}]\label{lem:identities}
    We have the following identities:
    \[
        \sk{n+1}( \bdcube{k} \hookrightarrow \cube{k}) =
        \begin{cases}
            \bdcube{k} \hookrightarrow \cube{k} & \text{if } k \leq n+1 \\
            \mathrm{id}_{\sk{n+1}\cube{k}} & \text{if } k \geq n+2
        \end{cases}
    \]
    and
    \[
        \sk{n+1}( \obox{k}{i,\varepsilon} \hookrightarrow \cube{k}) =
        \begin{cases}
            \obox{k}{i,\varepsilon} \hookrightarrow \cube{k} & \text{if } k \leq n+1 \\
            \obox{n+2}{i,\varepsilon} \hookrightarrow \bdcube{n+2} & \text{if } k = n+2 \\
            \mathrm{id}_{\sk{n+1}\cube{k}} & \text{if } k \geq n+3.
        \end{cases}
    \]
\end{lemma}
\begin{proof}
    These identities can be verified using skeletal induction for cubical sets (see, for example, \cite[Thm.~1.3.8]{cisinski-book}).
\end{proof}

\begin{proposition}[cf.~\cref{prop:transfer-naive-n-fibrations-simplicial}]\label{prop:transfer-naive-n-fibrations}
    A cubical map $f \from X \to Y$ is a transferred naive $n$-fibration if and only if it has the right lifting property with respect to the set
    \[
        J_{n}' = \left\{ \obox{k}{i,\varepsilon} \hookrightarrow \cube{k} \ \mid \ 0 < k \leq n+1, \ 1 \leq i \leq k, \ \varepsilon = 0, 1 \right\} \cup \left\{ \obox{n+2}{i,\varepsilon} \hookrightarrow \bdcube{n+2} \ \mid \ 1 \leq i \leq n+2, \ \varepsilon = 0, 1  \right\}.
    \]
\end{proposition}
\begin{proof}
    Observe that given two cubical maps $i \from A \to B$ and $f \from X \to Y$, $f$ has the right lifting property with respect to $\sk{n+1}i$ if and only if $\cosk{n+1}f$ has the right lifting property with respect to $i$.
    The required result then follows from \cref{lem:identities}.
\end{proof}

\begin{corollary}\label{cor:Kan-complexes-are-n-fibrant}
    Every Kan complex is a transferred $n$-fibrant cubical set.
\end{corollary}
\begin{proof}
    Let $X$ be a Kan complex.
    Then, for $k > 0$, $0 \leq i \leq k$, and $\varepsilon = 0,1$, every map $\obox{k}{i,\varepsilon} \to X$ can be extended to  a map $\cube{k} \to X$.
    For $k = n+2$, restricting the map $\cube{n+2} \to X$ to $\bdcube{n+2}$ allows us to extend any map $\obox{n+2}{i,\varepsilon} \to X$ to $\bdcube{n+2} \to X$.
    Thus, the unique map $X \to \cube{0}$ is a transferred naive $n$-fibration by \cref{prop:transfer-naive-n-fibrations}.
\end{proof}

The above result will be key to our construction of the fibration category structure for the discrete $n$-types of graphs in \cref{sec:graphs}.

We also have the following characterization of transferred $n$-cofibrations.

\begin{proposition}[cf.~\cref{prop:transfer-n-cofibrations-simplicial}]\label{prop:transfer-n-cofibrations}
    A cubical map $f \from X \to Y$ is both a transferred $n$-fibration and an $n$-equivalence if and only if it has the right lifting property with respect to the set
    \[
        I_{n}' = \left\{ \bdcube{k} \hookrightarrow \cube{k} \ \mid \ 0 \leq k \leq n+1 \right\}.
    \]
    Equivalently, a cubical map $i \from A \to B$ is a transferred $n$-cofibration if and only if it belongs to the saturation of $I_{n}'$.
\end{proposition}
\begin{proof}
    By \cref{prop:transfer-n-equivalences}, a map $f \from X \to Y$ is both a transferred $n$-fibration and an $n$-equivalence if and only if
    $\cosk{n+1}f \from \cosk{n+1}X \to \cosk{n+1}Y$ is both an $n$-fibration and an $n$-equivalence.
    By \cref{thm:cisinski-model-cubical}, we know that $\cosk{n+1}f \from \cosk{n+1}X \to \cosk{n+1}Y$ is both an $n$-fibration and an $n$-equivalence
    if and only if it has the right lifting property with respect to the set
    \[
        I = \left\{ \bdcube{k} \hookrightarrow \cube{k} \ \mid \ k \geq 0 \right\}
    \]
    of all boundary inclusions, and hence with respect to every monomorphism.
    Since $\cosk{n+1}f$ has the right lifting property with respect to a cubical map $i \from A \to B$ if and only if $f$ has the right lifting property with respect to $\sk{n+1}i$, the required result then follows from \cref{lem:identities}.
\end{proof}

We will need the following lemma in order to verify the acyclicity condition in \cite[Cor.~3.3.4]{hess-kedziorek-riehl-shipley}.

\begin{lemma}[cf.~\cref{lem:counit-skeleton-simplicial}]\label{lem:counit-skeleton}
    For every cubical set $X$, the canonical map $\varepsilon_{X} \from \sk{n+1}X \to X$ is $n$-anodyne.
\end{lemma}
\begin{proof}
    The map $\varepsilon_{X} \from \sk{n+1}X \to X$ is a transfinite composite of the following sequence:
    \[
    \begin{tikzcd}
        \sk{n+1}X \arrow[r] & \sk{n+2}X \arrow[r] & \sk{n+3}X \arrow[r] & \cdots
    \end{tikzcd}
    \]
    and each map in this sequence is $n$-anodyne (see \cite[Thm.~1.3.8]{cisinski-book}).
\end{proof}

\begin{lemma}[cf.~\cref{lem:acyclicity-condition-simplicial}]\label{lem:acyclicity-condition}
    Let $i \from A \to B$ be a cubical map with left lifting property with respect to all transferred $n$-fibrations.
    Then, $i$ is an $n$-equivalence.
\end{lemma}

\begin{proof}
    From the definitions, it follows that a cubical map $f \from X \to Y$ is a transferred $n$-fibration if and only it has the right lifting property against $\sk{n+1}j$ for every map $j \from C \to D$ that is both a monomorphism and an $n$-equivalence.
    Consider the naturality square:
    \[
    \begin{tikzcd}
        \sk{n+1}C \arrow[r,"\varepsilon_{C}"] \arrow[d,"\sk{n+1}j"'] & C \arrow[d,"j"]  \\
        \sk{n+1}D \arrow[r,"\varepsilon_{D}"] & D
    \end{tikzcd}
    \]
    By \cref{lem:counit-skeleton}, both the horizontal maps in the above diagram are $n$-anodyne, and hence, both monomorphisms and $n$-equivalences (see \cite[Prop.~2.4.25]{cisinski-book}).
    Thus, if $j$ is a monomorphism and an $n$-equivalence, so is $\sk{n+1}j$.
    Since the class of maps that are both monomorphisms and $n$-equivalences is saturated (see \cref{thm:cisinski-model-cubical}), it follows that any map $i \from A \to B$ that has the left lifting property with respect to all transferred $n$-fibrations is both a monomorphism and an $n$-equivalence.
    That is, we have the following:
    \begin{align*}
        \mathscr{L}{(\text{transferred }n\text{-fibrations})} &= \mathscr{L}(\mathscr{R}(\mathrm{sk}_{n+1}(\text{monos} \cap n\text{-equivalences}))) \\
        &\subseteq \mathscr{L}(\mathscr{R}(\text{monos} \cap n\text{-equivalences})) \\
        &= \text{monos} \cap n\text{-equivalences} \\
        &\subseteq n\text{-equivalences}. \qedhere
    \end{align*}
\end{proof}

\begin{proof}[Proof of \cref{thm:transferred-model-cubical}]
    Since the category of cubical sets is a presheaf category, and thus accessible, \cref{lem:acyclicity-condition} allows us to apply \cite[Cor.~3.3.4]{hess-kedziorek-riehl-shipley} to right-transfer the Cisinski model category structure for $n$-types of cubical sets (see \cref{thm:cisinski-model-cubical}) along the adjunction $\sk{n+1} \dashv \cosk{n+1}$.
    This proves the existence of the transferred model category structure and also tells us that this adjunction is Quillen.
    By \cref{prop:transfer-n-equivalences}, we know that $\cosk{n+1}$ preserves and reflects $n$-equivalences.
    Thus, it suffices to prove that for every cubical set $X$, the unit $X \to \cosk{n+1}\sk{n+1}X$ of the adjunction $\sk{n+1} \dashv \cosk{n+1}$ is an $n$-equivalence.
    Observe that the functor $i_{n+1}^{*} \circ {i_{n+1}}_{*} \from \cSet[\leq n+1] \to \cSet[\leq n+1]$ equals the identity functor on $\cSet[\leq n+1]$.
    It follows that
    \[
        \cosk{n+1} \circ \sk{n+1} = {i_{n+1}}_{!} \circ i_{n+1}^{*} \circ {i_{n+1}}_{*} \circ i_{n+1}^{*} = {i_{n+1}}_{!} \circ i_{n+1}^{*} = \cosk{n+1}.
    \]
    The required result then follows from \cref{lem:unit-coskeleton}.
\end{proof}

Having constructed the cubical analogue of the Elvira--Hernandez model category structure of \cref{thm:transferred-model-simplicial}, we now take our cue from \cref{remark:transferring-makes-naivete-redundant-simplicial} where we saw that, in the simplicial setting, every transferred naive $n$-fibration is a transferred $n$-fibration, and vice-versa.
It is straightforward to see, even in the cubical setting, that every transferred $n$-fibration is a transferred naive $n$-fibration.
This follows directly from the the definitions and the observation that every $n$-anodyne map is both a monomorphism and an $n$-equivalence (see \cite[Prop.~2.4.25]{cisinski-book}).
However, proving that every transferred naive $n$-fibration is a transferred $n$-fibration requires some more machinery, viz. the theory of minimal fibrations as developed in \cite[\S~5.1]{cisinski-book} and specifically applied to the Grothendieck model category structure of \cref{thm:grothendieck}.
We briefly recall the definition and some important properties.

\begin{definition}
\leavevmode
\begin{enumerate}
    \item Let $p \from X \to Y$ be a map in $\cSet$ and let $u, v \from \cube{k} \to X$ be two $k$-cubes such that $p \circ u = p \circ v$ and $u \vert_{\bdcube{k}} = v|_{\bdcube{k}}$.
    We say that $u$ is \emph{fiberwise homotopic} to $v$ \emph{mod} $\bdcube{k}$, written $u \sim_{p} v \mod \bdcube{k}$, if there exists a homotopy $H \from \cube{k} \otimes \cube{1} \to X$ from $u$ to $v$ such that the following diagrams commute:
    \[
    \begin{tikzcd}
        \cube{k} \otimes \cube{1} \arrow[r,"H"] \arrow[d] & X \arrow[d,"p"] \\
        \cube{k} \arrow[r,"p \circ u = p \circ v"] & Y
    \end{tikzcd}
    \qquad
    \text{and}
    \qquad
    \begin{tikzcd}
        \bdcube{k} \otimes \cube{1} \arrow[r,hook] \arrow[d] & \cube{k} \otimes \cube{1} \arrow[d,"H"] \\
        \bdcube{k} \arrow[r,"u \mid_{\bdcube{k}} = v \mid_{\bdcube{k}}"] & X\text{.}
    \end{tikzcd}
    \]
    \item A \emph{minimal Kan fibration} $p \from X \to Y$ is a Kan fibration (in the sense of \cref{thm:grothendieck}) such that, for any two $k$-cubes $u, v \from \cube{k} \to X$, $k \geq 0$, such that $p \circ u = p \circ v$ and $u \mid_{\bdcube{k}} = v \mid_{\bdcube{k}}$, if we have $u \sim_{f} v \mod \bdcube{k}$, then we must have $u = v$.
\end{enumerate}
\end{definition}

\begin{proposition}[\protect{\cite[Prop.~5.1.15]{cisinski-book}}] \label{prop:minimal-fibrations-stable-under-pullbacks}
    The class of minimal Kan fibrations in $\cSet$ is stable under pullbacks.
    \qed
\end{proposition}

\begin{theorem}[\protect{\cite[Thm.~5.1.17]{cisinski-book}}] \label{thm:existence-of-minimal-models}
    Every Kan fibration $p \from X \to Y$ in $\cSet$ can be factored as trivial Kan fibration $r \from X \to S$ followed by a minimal Kan fibration $q \from S \to Y$.
    \qed
\end{theorem}

\begin{proposition}[\protect{\cite[Prop.~5.1.19]{cisinski-book}}] \label{prop:weak-eqs-between-two-minimal-fibrations}
    Given any commutative triangle of the form
    \[
    \begin{tikzcd}[column sep = small]
        X \arrow[rr,"f"] \arrow[dr,"p"'] && X' \arrow[dl,"p'"] \\
        & Y &
    \end{tikzcd}
    \]
    where $p$ and $p'$ are both minimal Kan fibrations, if $f$ is a weak homotopy equivalence, then it must be an isomorphism.
    \qed
\end{proposition}

In addition to the above results, we will also require the following lemma about minimal Kan fibrations.

\begin{lemma} \label{lem:extending-minimal-Kan-fibrations}
    Let $p \from X \to Y$ be a minimal Kan fibration and let $i \from Y \to Z$ be a trivial cofibration.
    Then, there exists a pullback square of the form
    \[
    \begin{tikzcd}
        X \arrow[r,"j"] \arrow[d,"p"'] \arrow[dr,phantom,"\lrcorner" description, very near start] & W \arrow[d,"q"] \\
        Y \arrow[r,"i"] & Z
    \end{tikzcd}
    \]
    where $j \from X \to W$ is a trivial cofibration and $q \from W \to Z$ is a Kan fibration.
\end{lemma}
\begin{proof}
    We can factor the composite $i \circ p \from X \to Z$ as a trivial cofibration $j' \from X \to W'$ followed by a Kan fibration $q' \from W' \to Z$.
    By \cref{thm:existence-of-minimal-models}, we can further factor $q' \from W' \to Z$ as a trivial fibration $r \from W' \to W$ followed by a minimal Kan fibration $q \from W \to Z$.
    Setting $j = r \circ j'$, we have the following commutative square:
    \[
    \begin{tikzcd}
        X  \arrow[r,"j"] \arrow[d,"p"'] & W \arrow[d,"q"] \\
        Y \arrow[r,"i"] & Z
    \end{tikzcd}
    \]
    Consider the pullback of $q \from W \to Z$ and $i \from Y \to Z$.
    Since $q$ is a Kan fibration and $i$ is a weak homotopy equivalence, it follows from the right properness of the Grothendieck model category structure on $\cSet$ that the map $q^{*}i \from Y \times_{Z} W \to W$ is a weak homotopy equivalence.
    The map $j \from X \to W$, which is a composite of the trivial cofibration $j'$ followed by the trivial fibration $r$, is also a weak homotopy equivalence.
    Thus, by 2-out-of-3, the induced map $X \to Y \times_{Z} W$ is also a weak homotopy equivalence.
    Since both $p \from X \to Y$ and $i^{*}q \from Y \times_{Z} W$ are minimal Kan fibrations (the first by hypothesis and the second by \cref{prop:minimal-fibrations-stable-under-pullbacks}).
    By \cref{prop:weak-eqs-between-two-minimal-fibrations}, the induced map $X \to Y \times_{Z} W$ must be an isomorphism.
    That is, the above commutative square is, in fact, a pullback square.
    Finally, since monomorphisms are preserved under pullbacks, $j \from X \to W$ is a monomorphism in addition to being a weak homotopy equivalence, i.e., $j$ is a trivial cofibration.
\end{proof}

We now return to the task of proving that every transferred naive $n$-fibration is, in fact, a transferred $n$-fibration.
We proceed in three steps, the first of which is a technical lemma.

\begin{lemma}
\label{lem:keystone-lemma}
    Let $X, Y \in \cSet$ be Kan complexes, and $f \from X \to Y$ be both a Kan fibration and an $n$-equivalence.
    Then, $f$ has the right lifting property with respect to the set
    \[
        I_{n}' = \left\{\bdcube{k} \hookrightarrow \cube{k} \ \mid \ 0 \leq k \leq n+1\right\}\text{.}
    \]
\end{lemma}
\begin{proof}
    \cite[Lem.~4.9]{carranza-kapulkin:homotopy-groups} tells us that $f$ has the right lifting property with respect to $\varnothing \hookrightarrow \cube{0}$.
    Let $0 < k \leq n+1$, and consider a lifting problem as follows:
    \[
    \begin{tikzcd}
        \bdcube{k} \arrow[r,"u"] \arrow[d,hook] & X \arrow[d,"f"] \\
        \cube{k} \arrow[r,"v"] & Y.
    \end{tikzcd}
    \]
    \cite[Thm.~3.22]{carranza-kapulkin:homotopy-groups} tells us that the map $u \from \bdcube{k} \to X$ represents an element in the kernel of the group homomorphism $f_{*} \from \pi_{k-1}(X,x) \to \pi_{k-1}(Y,fx)$, where $x = u(0,\ldots,0)$.
    Since $f_{*}$ is an isomorphism by hypothesis, $u$ must represent the trivial element in $\pi_{k-1}(X,x)$.
    Thus, there exist maps $H \from \bdcube{k} \otimes \cube{1} \to X$ and $G \from \cube{k} \otimes \cube{1} \to Y$ such that $H\partial^{0}_{\bdcube{k}} = u$, $H\partial^{1}_{\bdcube{k}} = \mathrm{const}_{x}$, $G\partial^{0}_{\cube{k}} = v$, $G\partial^{1}_{\cube{k}} = \mathrm{const}_{fx}$, and $f H = G \mid_{\bdcube{k} \otimes \cube{1}}$.
    Therefore, we can define the following lifting problem:
    \[
    \begin{tikzcd}
        \obox{k+1}{k,1} \cong (\bdcube{k} \otimes \cube{1}) \cup (\cube{k} \otimes \left\{0\right\}) \arrow[r,"{[H,v]}"] \arrow[d,hook] & X \arrow[d,"f"] \\
        \cube{k+1} \cong \cube{k} \otimes \cube{1} \arrow[r,"G"] \arrow[ur,dotted] & Y
    \end{tikzcd}
    \]
    Since $f$ is a Kan fibration, the above lifting problem admits a solution $s \from \cube{k} \otimes \cube{1} \to X$.
    Then, $s \partial^{0}_{\cube{k}} \from \cube{k} \to X$ is a solution to the original lifting problem.
\end{proof}

Equipped with the above lemma and the theory of minimal Kan fibrations recalled before that, we can prove a special case of the general result we seek.

\begin{proposition}
\label{prop:transfer-n-fibrations-v1}
    Suppose $f \from X \to Y$ is both a transferred naive $n$-fibration and an $n$-equivalence.
    Then, $f$ is also a transferred $n$-fibration.
\end{proposition}
\begin{proof}
    By \cref{def:transferred-model} and \cref{prop:transfer-n-equivalences}, $\cosk{n+1}f$ is both a naive $n$-fibration and an $n$-equivalence.
    In particular, $\cosk{n+1}f$ is a Kan fibration.
    By \cref{thm:existence-of-minimal-models}, we can factor $\cosk{n+1}f$ as a trivial Kan fibration $r \from \cosk{n+1}E \to S$ followed by a minimal Kan fibration $q \from S \to \cosk{n+1}B$.
    By 2-out-of-3, $q$ must also be an $n$-equivalence.
    Let $\cosk{n+1}B \to \tilde{\cosk{n+1}B}$ be the fibrant replacement of $\cosk{n+1}B$ in the Grothendieck model category structure on $\cSet$.
    By \cref{lem:extending-minimal-Kan-fibrations}, there exists a pullback square of the form
    \[
    \begin{tikzcd}
        S \arrow[r] \arrow[d,"q"'] \arrow[dr,phantom,"\lrcorner" description, very near start] & \tilde{S} \arrow[d,"\tilde{q}"] \\
        \cosk{n+1}B \arrow[r] & \tilde{\cosk{n+1}B}
    \end{tikzcd}
    \]
    where $S \to \tilde{S}$ is a trivial cofibration and $\tilde{q} \from \tilde{S} \to \tilde{\cosk{n+1}B}$ is a Kan fibration.
    By 2-out-of-3, $\tilde{q}$ must be an $n$-equivalence.
    By \cref{lem:keystone-lemma}, $\tilde{q}$ has the right lifting property with respect to $I_{n}'$.
    
    Since the class $\mathscr{R}(I_{n}')$ of all maps having the right lifting property with respect to $I_{n}'$ is stable under pullbacks, $q$ is also contained in it.
    The map $r \from \cosk{n+1}E \to S$ is a trivial Kan fibration and so, is contained in $\mathscr{R}(I_{n}')$.
    Thus, the composite $\cosk{n+1}f = q \circ r$ is also contained in the class $\mathscr{R}(I_{n}')$, which implies, by transposition and \cref{lem:identities}, that $f$ is also contained in the class $\mathscr{R}(I_{n}')$.
    We then use \cref{prop:transfer-n-cofibrations} to conclude that $p$ is a transferred $n$-fibration, in addition to being an $n$-equivalence.
\end{proof}

We can finally prove the general result.

\begin{proposition}\label{prop:transferring-makes-naivete-redundant}
    Suppose $i \from A \to B$ is both a transferred $n$-cofibration and an $n$-equivalence.
    Then, it is contained in the saturation of
    \[
        J_{n}' = \left\{ \obox{k}{i,\varepsilon} \hookrightarrow \cube{k} \ \mid \ 0 < k \leq n+1, \ 1 \leq i \leq k, \ \varepsilon = 0, 1 \right\} \cup \left\{ \obox{n+2}{i,\varepsilon} \hookrightarrow \bdcube{n+2} \ \mid \ 1 \leq i \leq n+2, \ \varepsilon = 0, 1  \right\}\text{.}
    \]
    Equivalently, every transferred naive $n$-fibration is a transferred $n$-fibration.
\end{proposition}
\begin{proof}
    Applying the small object argument to $J_{n}'$, we can factor $i \from A \to B$ as a map $j \from A \to E$ belonging to the saturation of $J_{n}'$ followed by a transferred naive $n$-fibration $p \from E \to B$.
    Note that $j \from A \to E$ is both a monomorphism and an $n$-equivalence.
    By 2-out-of-3, $p \from E \to B$ is also an $n$-equivalence.
    By \cref{prop:transfer-n-fibrations-v1}, $p$ is a transferred $n$-fibration and hence, has the right lifting property with respect to $i$.
    Thus, by the usual retract argument, $i$ must be a retract of $j$ and so, it is also contained in the saturation of $J_{n}'$.
\end{proof}

\section{Application to discrete homotopy theory} \label{sec:graphs}

In this section, we use the results of the previous section to construct a fibration category structure on the category of graphs for discrete homotopy $n$-types of graphs and prove that for $n=1$, it is equivalent to the fibration category structure on the category of groupoids.
Fibration categories, introduced by Brown \cite{brown} and studied extensively by Radulescu-Banu \cite{radulescu-banu} and Szumi{\l}o \cite{szumilo-cofibration-categories} are a weakening of the notion of a model category designed to model homotopy theories with finite limits \cite{szumilo-complete-quasicategories}.

We begin by reviewing the requisite background on discrete homotopy theory, including the construction of the nerve functor from graphs to cubical sets.
After that, we prove the existence of the fibration category structure  for discrete homotopy $n$-types (\cref{thm:fib-cat-Graph-n}), and show that, for $n=1$, it is equivalent to the fibration category of groupoids (\cref{thm:fund-gpd-is-a-weak-equiv}).

\subsection{Background on discrete homotopy theory}

We begin by introducing graphs and graph maps.

\begin{definition} \leavevmode
\begin{enumerate}
    \item A \emph{graph} $X = (\V{X},\E{X})$ is a set $\V{X}$ equipped with a reflexive and symmetric relation $\E{X} \subseteq \V{X} \times \V{X}$.
    Elements of $\V{X}$ are called \emph{vertices} of the graph, and two vertices $x$ and $x'$ in $X$ are connected by an \emph{edge} if $(x,x') \in \E{X}$, in which case we write $x \sim x'$ and say $x$ and $x'$ are adjacent.
    In our figures, we suppress the unique loop present at each vertex, and represent each pair of oppositely directed edges by a single, undirected edge.

    \item A \emph{graph map} $f \from X \to Y$ is a set-function $f \from \V{X} \to \V{Y}$ that preserves the relation.
    In particular, since our graphs are reflexive, this allows edges in the domain to be collapsed to a single vertex in the codomain.
    We write $\Graph$ for the the category of graphs and graph maps.
\end{enumerate}
\end{definition}

\begin{example}\label{def:graph-examples} \leavevmode
\begin{enumerate}
    \item The \emph{finite interval} graph $I_{n}$ of length $n \in \mathbb{N}$ has, as vertices, the integers $0, 1, \ldots, n$, and an edge $i \sim j$ whenever $\lvert i - j \rvert \leq 1$.
    
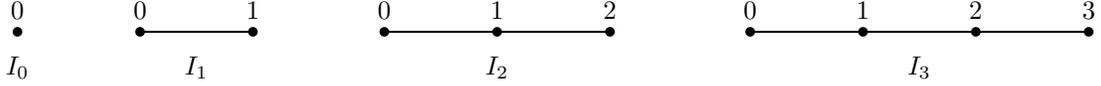
\begin{figure}[H]
    \centering
    \begin{minipage}{0.09\textwidth}
    \centering
    \begin{tikzpicture}[colorstyle/.style={circle, draw=black, fill=black, thick, inner sep=0pt, minimum size=1 mm, outer sep=0pt},scale=1.5]
        \node (0) at (0,0) [colorstyle, label = above: {$0$}] {};
        \node at (0,-0.5) [anchor = south]{$I_{0}$};
    \end{tikzpicture}
    \end{minipage}
    \begin{minipage}{0.19\textwidth}
    \centering
    \begin{tikzpicture}[colorstyle/.style={circle, draw=black, fill=black, thick, inner sep=0pt, minimum size=1 mm, outer sep=0pt},scale=1.5]
        \node (0) at (0,0) [colorstyle, label = above: {$0$}] {};
        \node (1) at (1,0) [colorstyle, label = above: {$1$}] {};
        \draw [thick] (0) -- (1);
        \node at (0.5,-0.5) [anchor = south]{$I_{1}$};
    \end{tikzpicture}
    \end{minipage}
    \begin{minipage}{0.29\textwidth}
    \centering
    \begin{tikzpicture}[colorstyle/.style={circle, draw=black, fill=black, thick, inner sep=0pt, minimum size=1 mm, outer sep=0pt},scale=1.5]
        \node (0) at (0,0) [colorstyle, label = above: {$0$}] {};
        \node (1) at (1,0) [colorstyle, label = above: {$1$}] {};
        \node (2) at (2,0) [colorstyle, label = above: {$2$}] {};
        \draw [thick] (0) -- (2);
        \node at (1,-0.5) [anchor = south]{$I_{2}$};
    \end{tikzpicture}
    \end{minipage}
    \begin{minipage}{0.39\textwidth}
    \centering
    \begin{tikzpicture}[colorstyle/.style={circle, draw=black, fill=black, thick, inner sep=0pt, minimum size=1 mm, outer sep=0pt},scale=1.5]
        \node (0) at (0,0) [colorstyle, label = above: {$0$}] {};
        \node (1) at (1,0) [colorstyle, label = above: {$1$}] {};
        \node (2) at (2,0) [colorstyle, label = above: {$2$}] {};
        \node (3) at (3,0) [colorstyle, label = above: {$3$}] {};
        \draw [thick] (0) -- (3);
        \node at (1.5,-0.5) [anchor = south]{$I_{3}$};
    \end{tikzpicture}
    \end{minipage}
    \end{figure}
    
    \item The \emph{infinite interval} graph $I_{\infty}$ has, as vertices, all integers $i \in \mathbb{Z}$, and an edge $i \sim j$ whenever $\lvert i - j \rvert \leq 1$.
    
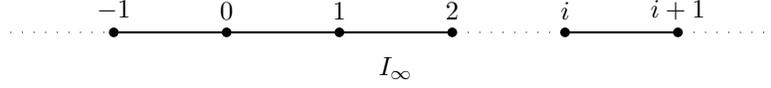
\begin{figure}[H]
    \centering
    \begin{minipage}{0.9\textwidth}
    \centering
    \begin{tikzpicture}[colorstyle/.style={circle, draw=black, fill=black, thick, inner sep=0pt, minimum size=1 mm, outer sep=0pt},scale=1.5]
       \node (-etc) at (0,0) {};
       \node (-1) at (1,0) [colorstyle, label = above: {$-1$}] {};
       \node (0) at (2,0) [colorstyle, label = above: {$0$}] {};
       \node (1) at (3,0) [colorstyle, label = above: {$1$}] {};
       \node (2) at (4,0) [colorstyle, label = above: {$2$}] {};
       \node (i) at (5,0) [colorstyle, label = above: {$i$}] {};
       \node (i+1) at (6,0) [colorstyle, label = above: {$i+1$}] {};
       \node (etc) at (7,0) {};
       \draw [loosely dotted] (-etc) -- (-1);
       \draw [thick] (-1) -- (0);
       \draw [thick] (0) -- (1);
       \draw [thick] (1) -- (2);
       \draw [loosely dotted] (2) -- (i);
       \draw [thick] (i) -- (i+1);
       \draw [loosely dotted] (i+1) -- (etc);
       \node at (3.5,-0.5) [anchor = south]{$I_{\infty}$};
   \end{tikzpicture}
   \end{minipage}
   \end{figure}
    \item The \emph{$n$-cycle} graph $C_{n}$ has, as vertices, the integers mod $n$, and an edge $i \sim j$ whenever $\lvert i - j \rvert \leq 1$.
    
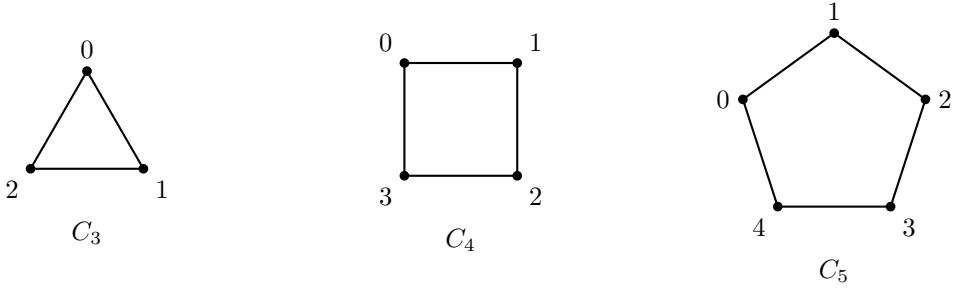
\begin{figure}[H]
    \centering
    \begin{minipage}{0.3\textwidth}
    \centering
    \begin{tikzpicture}[colorstyle/.style={circle, draw=black, fill=black, thick, inner sep=0pt, minimum size=1 mm, outer sep=0pt},scale=1.5]
        \node (0) at (0.5,0.866) [colorstyle, label = above: {$0$}] {};
        \node (1) at (1,0) [colorstyle, label = below right: {$1$}] {};
        \node (2) at (0,0) [colorstyle, label = below left: {$2$}] {};
        \draw [thick] (0) -- (1);
        \draw [thick] (1) -- (2);
        \draw [thick] (2) -- (0);
        \node at (0.5,-0.75) [anchor = south]{$C_{3}$};
    \end{tikzpicture}
    \end{minipage}
    \begin{minipage}{0.3\textwidth}
    \centering
    \begin{tikzpicture}[colorstyle/.style={circle, draw=black, fill=black, thick, inner sep=0pt, minimum size=1 mm, outer sep=0pt},scale=1.5]
        \node (0) at (0,1) [colorstyle, label = above left: {$0$}] {};
        \node (1) at (1,1) [colorstyle, label = above right: {$1$}] {};
        \node (2) at (1,0) [colorstyle, label = below right: {$2$}] {};
        \node (3) at (0,0) [colorstyle, label = below left: {$3$}] {};
        \draw [thick] (0) -- (1);
        \draw [thick] (1) -- (2);
        \draw [thick] (2) -- (3);
        \draw [thick] (3) -- (0);
        \node at (0.5,-0.75) [anchor = south]{$C_{4}$};
    \end{tikzpicture}
    \end{minipage}
    \begin{minipage}{0.3\textwidth}
    \centering
    \begin{tikzpicture}[colorstyle/.style={circle, draw=black, fill=black, thick, inner sep=0pt, minimum size=1 mm, outer sep=0pt},scale=1.5]
        \node (0) at (0.809,1.539) [colorstyle, label = above : {$1$}] {};
        \node (1) at (1.618,0.951) [colorstyle, label = right: {$2$}] {};
        \node (2) at (1.309,0) [colorstyle, label = below right: {$3$}] {};
        \node (3) at (0.309,0) [colorstyle, label = below left: {$4$}] {};
        \node (4) at (0,0.951) [colorstyle, label = left: {$0$}] {};
        \draw [thick] (0) -- (1);
        \draw [thick] (1) -- (2);
        \draw [thick] (2) -- (3);
        \draw [thick] (3) -- (4);
        \draw [thick] (4) -- (0);
        \node at (0.809,-0.75) [anchor = south]{$C_{5}$};
    \end{tikzpicture}
    \end{minipage}
    \end{figure}
\end{enumerate}
\end{example}

\begin{definition}\label{def:box-product}
    The \emph{box product} $X \gtimes Y$ of two graphs $X$ and $Y$ is defined as follows:
    \begin{align*}
        \V{(X \gtimes Y)} &= \V{X} \times \V{Y} \\
        \E{(X \gtimes Y)} &= \left\{ \left(x,y\right) \sim \left(x',y'\right) \ \middle\vert \
        \begin{aligned}
            \text{either } & x \sim x' \text{ and }    y = y' \\
            \text{or } & x = x' \text{ and } y \sim y'
        \end{aligned} \right\}
    \end{align*}
    This defines a monoidal structure on $\Graph$.
\end{definition}

\begin{definition}\leavevmode
\begin{enumerate}
    \item Given two graph maps $f, g \from X \to Y$, an \emph{$A$-homotopy} $H \from f \Rightarrow g$ of length $n \in \mathbb{N}$ is a map $H \from X \gtimes I_{n} \to Y$ such that $H\left(-,0\right) = f$ and $H\left(-,n\right) = g$.
    When such an $A$-homotopy exists, we say that $f$ and $g$ are $A$-homotopic and write $f \sim_{A} g$.
    This defines an equivalence relation on the set $\Graph\left(X,Y\right)$.
    \item A graph map $f \from X \to Y$ is an \emph{$A$-homotopy equivalence} if there exists a graph map $g \from Y \to X$ such that $g \circ f \sim_{A} \mathrm{id}_{X}$ and $f \circ g \sim_{A} \mathrm{id}_{Y}$.
\end{enumerate}
\end{definition}

\begin{example}
    The unique map $I_{n} \to I_{0}$ is an $A$-homotopy equivalence for every $n \in \mathbb{N}$, whereas the unique map $I_{\infty} \to I_{0}$ is not.
\end{example}

\begin{definition}
For $n > 0$, the \emph{$n$th discrete homotopy group} $A_{n}(X,x)$ of a pointed graph $(X,x)$ is given by:
\[
    A_{n}(X,x) = \left\{ f \from \gexp{I_{\infty}}{n} \to X \ \mid \ f(\mathbf{i}) = x \text{ for all but finitely many } \mathbf{i} \in \gexp{I_{\infty}}{n}\right\} / \sim_{\ast}
\]
where we are quotienting by $A$-homotopies that also take the value $x$ for all but finitely many vertices in $\gexp{I_{\infty}}{n+1}$.
The group operation is given by concatenating the finite non-constant regions.
\end{definition}

\begin{definition}
    A graph map $f \from X \to Y$ is a \emph{weak $A$-homotopy equivalence} if it induces a bijection $f_{*} \from \pi_{0}X \to \pi_{0}Y$ between the path components of $X$ and $Y$, and an isomorphism $f_{*} \from A_{n}(X,x) \to A_{n}(Y,fx)$ for every $n > 0$ and every vertex $x \in X$.
\end{definition}

\begin{example}
    It can be shown that all cycle graphs $C_{n}$ are weakly $A$-homotopy equivalent for $n \geq 5$.
    Furthermore, the unique map $I_{\infty} \to I_{0}$ is a weak $A$-homotopy equivalence.
\end{example}

\begin{definition} \leavevmode
\begin{enumerate}
    \item Given a graph $X$ and an integer $n \geq 0$, a graph map $f \from \gexp{I_{\infty}}{n} \to X$ is \emph{stable in all directions} if there exists an integer $M \geq 0$ such that for each $i = 1,\ldots, n$ and $\varepsilon = 0,1$ we have:
    \[
        f(t_{1},\ldots,t_{i-1},(2\varepsilon-1)t_{i},t_{i+1},\ldots,t_{n}) = f(t_{1},\ldots,t_{i-1},(2\varepsilon-1)M,t_{i+1},\ldots,t_{n})
    \]
    whenever $t_{i} \geq M$.
    \item The \emph{nerve} $\gnerve X$ of a graph $X$ is a cubical set whose $n$-cubes are given by:
    \[
        (\gnerve X)_{n} = \left\{f \from \gexp{I_{\infty}}{n} \to X \ \mid \ f \text{ is stable in all directions} \right\}
    \]
    and whose cubical operators are given as follows:
    \begin{itemize}
        \item the map $\partial_{i,\varepsilon} \from (\gnerve X)_{n} \to (\gnerve X)_{n-1}$ for $i = 1, \ldots, n$ and $\varepsilon = 0,1$ is given by
        \[
            f\partial_{i,\varepsilon} (t_{1},\ldots,t_{n-1}) = f(t_{1},\ldots,t_{i-1},(2\varepsilon-1)M,t_{i},\ldots,t_{n-1});
        \]
        where $M \geq 0$ is some integer such that for each $i = 1,\ldots, n$ and $\varepsilon = 0,1$ we have:
        \[
            f(t_{1},\ldots,t_{i-1},(2\varepsilon-1)t_{i},t_{i+1},\ldots,t_{n}) = f(t_{1},\ldots,t_{i-1},(2\varepsilon-1)M,t_{i+1},\ldots,t_{n})
        \]
        whenever $t_{i} \geq M$.
        \item the map $\sigma_{i} \from (\gnerve X)_{n} \to (\gnerve X)_{n+1}$ for $i = 1, \ldots, n+1$ is given by
        \[
            f\sigma_{i} (t_{1},\ldots,t_{n+1}) = f(t_{1},\ldots,t_{i-1},t_{i+1},\ldots,t_{n+1});
        \]
        \item the map $\gamma_{i,\varepsilon} \from (\gnerve X)_{n+1} \to (\gnerve X)_{n}$ for $i = 1, \ldots, n$ and $\varepsilon = 0, 1$ is given by
        \[
            f\gamma_{i,\varepsilon} (t_{1},\ldots,t_{n+1}) =
            \begin{cases}
                f(t_{1},\ldots,t_{i-1},\max{(t_{i},t_{i+1})},t_{i+2},\ldots,t_{n+1}) & \text{if } \varepsilon = 0, \\
                f(t_{1},\ldots,t_{i-1},\min{(t_{i},t_{i+1})},t_{i+2},\ldots,t_{n+1}) & \text{if } \varepsilon = 1.
            \end{cases}
        \]
    \end{itemize}
    This defines the \emph{nerve functor} $\gnerve \from \Graph \to \cSet$.
\end{enumerate}
\end{definition}

\begin{theorem}\label{thm:cubical-nerve-of-graphs} \leavevmode
\begin{enumerate}
    \item The nerve functor $\gnerve \from \Graph \to \cSet$ preserves all finite limits.
    \item The nerve $\gnerve X$ of any graph $X$ is a Kan complex.
    \item We have a bijection $\pi_{0}X \cong \pi_{0} \gnerve X$ for every $X \in \Graph$ and an isomorphism $A_{n}(X,x) \cong \pi_{n}(\gnerve X,x)$ for every $(X,x) \in \Graph_{*}$ and every $n > 0$.
\end{enumerate}
\end{theorem}
\begin{proof}
    (1) is \cite[Prop.~3.8]{carranza-kapulkin:cubical-graphs}.
    (2) is \cite[Thm.~4.5]{carranza-kapulkin:cubical-graphs}.
    (3) is \cite[Thm.~4.6]{carranza-kapulkin:cubical-graphs}.
\end{proof}

Thus, we may view the nerve functor as taking values in the fibration category $\Kan$ of cubical Kan complexes.

\begin{definition}
    A graph map $f \from X \to Y$ is a \emph{fibration} if its image $\gnerve f \from \gnerve X \to \gnerve Y$ under the nerve functor $N \from \Graph \to \Kan$ is a Kan fibration of cubical sets.
\end{definition}

We now describe a useful criterion for verifying whether a graph map is a fibration.

\begin{notation}
\begin{enumerate}
    \item For integers $m, n \geq 0$, the graph $\gexp{I_{m}}{n}$ has vertices given by $n$-tuples $(x_{1},\ldots,x_{n})$ where each $x_{i} \in \left\{0,\ldots,m\right\}$ for $i = 1,\ldots,n$, and edges given by $(x_{1},\ldots,x_{i},\ldots,x_{n}) \sim (x_{1},\ldots,x_{i} \pm 1,\ldots,x_{n})$ for any $i = 1, \ldots,n$.
    \item For integers $m, n \geq 0$, let $\partial \gexp{I_{m}}{n}$ denote the full subgraph of $\gexp{I_{m}}{n}$ on the vertices $(x_{1},\ldots,x_{n})$ for which, there exists at least one $j \in \left\{1,\ldots,n\right\}$ such that either $x_{j} = 0$ or $x_{j} = m$.
    \item For integers $m, n > 0$, let $\obox{n}{m}$ denote the full subgraph of $\partial \gexp{I_{m}}{n}$ on the vertices $(x_{1},\ldots,x_{n})$ for which, if we have $x_{n} = m$, there must exist at least one $j \in \left\{1,\ldots,n-1\right\}$ such that either $x_{j} = 0$ or $x_{j} = m$.
    That is, $\obox{n}{m}$ is obtained from $\partial \gexp{I_{m}}{n}$, which already excludes the interior vertices of $\gexp{I_{m}}{n}$, by further removing the interior vertices of the face $x_{n} = m$.
    \item For integers $m, \delta \geq 0$, we have a graph map $c_{m,\delta} \from I_{m+2\delta} \to I_{m}$ given by the formula
    \[
        c_{m,\delta}(i) = \begin{cases}
            0 & \text{if } 0 \leq i \leq \delta \\
            i-\delta & \text{if } \delta \leq i \leq m+\delta \\
            m & \text{if } m+\delta \leq i \leq m+2\delta\text{.}
        \end{cases}
    \]
    Taking iterated box products, we obtain for each $n \geq 0$, a map $\gexp{c_{m,\delta}}{n} \from \gexp{I_{m+2\delta}}{n} \to \gexp{I_{m}}{n}$, which further induces maps $\partial\gexp{I_{m+2\delta}}{n} \to \partial\gexp{I_{m}}{n}$ and $\obox{n}{m+2\delta} \to \obox{n}{m}$.
\end{enumerate}
\end{notation}

\begin{proposition}[\protect{\cite[Prop.~5.11]{carranza-kapulkin:cubical-graphs}}]
    \label{prop:criterion-fibrations}
    A graph map $f \from X \to Y$ is a fibration if and only if, for any integers $m, n > 0$ and commutative square of the form
    \[
    \begin{tikzcd}
        \obox{n}{m} \arrow[r] \arrow[d,hook] & X \arrow[d,"f"] \\
        \gexp{I_{m}}{n} \arrow[r] & Y
    \end{tikzcd}
    \]
    there exists an integer $\delta \geq 0$ and a map $\gexp{I_{m+2\delta}}{n} \to X$ that makes the following diagram commute:
    \[
    \pushQED{\qed}
    \begin{tikzcd}
        \obox{n}{m+2\delta} \arrow[r] \arrow[d,hook] & \obox{n}{m} \arrow[r] \arrow[d,hook] & X \arrow[d,"f"] \\
        \gexp{I_{m+2\delta}}{n} \arrow[r,"\gexp{c_{m,\delta}}{n}"'] \arrow[urr,dotted] & \gexp{I_{m}}{n} \arrow[r] & Y
    \end{tikzcd}
    \qedhere
    \popQED
    \]
\end{proposition}

\begin{theorem}[\protect{\cite[Thm.~5.9,~Prop.~5.12,~Prop.~5.13]{carranza-kapulkin:cubical-graphs}}] \label{thm:fib-cat-Graph-infty}
    The category $\Graph$ admits a fibration category structure where
    \begin{itemize}
        \item the fibrations are maps which are sent to Kan fibrations under the nerve functor $\gnerve \from \Graph \to \Kan$, and
        \item the weak equivalences are the weak $A$-homotopy equivalences.
    \end{itemize}
    Given this structure on $\Graph$, the nerve functor $\gnerve \from \Graph \to \Kan$ is an exact functor of fibration categories.
    \qed
\end{theorem}

A central open problem in discrete homotopy theory is to determine whether the nerve functor $N \from \Graph \to \Kan$ is a weak equivalence of fibration categories.
Although we do not solve this problem here, in the remainder of this paper, we propose a new line of attack towards this problem, using the model category structure on $\cSet$ constructed in \cref{thm:transferred-model-cubical}.

\subsection{Homotopy $n$-types of graphs}

We fix an integer $n \geq 0$ for the remainder of this subsection.
We can now define (discrete) $n$-equivalences.

\begin{definition}
    A graph map $f \from X \to Y$ is an \emph{$n$-equivalence} if it induces a bijection $f_{*} \from \pi_{0}X \to \pi_{0}Y$ between the set of path-components, and an isomorphism $f_{*} \from A_{n}(X,x) \to A_{n}(Y,fx)$ for every $n > 0$ and every vertex $x \in X$.
\end{definition}

We construct a fibration category structure on $\Graph$ where the weak equivalences are the $n$-equivalences. 
It is unclear whether this fibration category structure can be upgraded to a full model category structure and one should expect a negative answer.
It is known, for example, that when considering discrete homotopy equivalences, as opposed to discrete weak or $n$-equivalences, no suitable model category structure exists \cite{carranza-kapulkin-kim}.

Note that the nerve of every graph is a transferred $n$-fibrant cubical set (since every Kan complex is transferred $n$-fibrant).
Thus, we may view the nerve functor as taking values in the fibration catgeory $\cSet[n]^{\mathrm{fib}}$ of transferred $n$-fibrant cubical sets.
This is why we need to work with the analogue of the model category structure of \cite{Elvira-Donazar_Hernandez-Paricio_1995} instead of the Bousfield localization of the Grothendieck model category structure.
The latter would require restricting the class of graphs to those that are homotopy $n$-types.
However constructing a homotopy $n$-type approximation of an arbitrary graph remains an open problem.

\begin{definition}
    A graph map $f \from X \to Y$ is an \emph{$n$-fibration} if its image $\gnerve f \from \gnerve X \to \gnerve Y$ under the nerve functor $\gnerve \from \Graph \to \cSet[n]^{\mathrm{fib}}$ is a transferred $n$-fibration.
    Or equivalently, by virtue of \cref{prop:transferring-makes-naivete-redundant}, if $\gnerve f$ has the right lifting property with respect to the set
    \[
        J_{n}' = \left\{ \obox{k}{i,\varepsilon} \hookrightarrow \cube{k} \ \mid \ 0 < k \leq n+1, \ 1 \leq i \leq k, \ \varepsilon = 0, 1 \right\} \cup \left\{ \obox{n+2}{i,\varepsilon} \hookrightarrow \bdcube{n+2} \ \mid \ 1 \leq i \leq n+2, \ \varepsilon = 0, 1  \right\}.
    \]
\end{definition}

Analogous to \cref{prop:criterion-fibrations}, we have the following criterion for verifying whether a graph map is an $n$-fibration.

\begin{proposition}
    A graph map $f \from X \to Y$ is an $n$-fibration if and only if, for any integers $m > 0$, $0 < k \leq n+1$ and commutative square of the form
    \[
    \begin{tikzcd}
        \obox{k}{m} \arrow[r] \arrow[d,hook] & X \arrow[d,"f"] \\
        \gexp{I_{m}}{k} \arrow[r] & Y
    \end{tikzcd}
    \]
    there exists an integer $\delta \geq 0$ and a map $\gexp{I_{m+2\delta}}{k} \to X$ that makes the following diagram commute:
    \[
    \begin{tikzcd}
        \obox{k}{m+2\delta} \arrow[r] \arrow[d,hook] & \obox{k}{m} \arrow[r] \arrow[d,hook] & X \arrow[d,"f"] \\
        \gexp{I_{m+2\delta}}{k} \arrow[r] \arrow[urr,dotted] & \gexp{I_{m}}{k} \arrow[r] & Y
    \end{tikzcd}
    \]
    and furthermore, for any integer $m > 0$ and commutative square of the form
    \[
    \begin{tikzcd}
        \obox{n+2}{m} \arrow[r] \arrow[d,hook] & X \arrow[d,"f"] \\
        \partial\gexp{I_{m}}{n+2} \arrow[r] & Y
    \end{tikzcd}
    \]
    there exists an integer $\delta \geq 0$ and a map $\partial \gexp{I_{m+2\delta}}{n+2} \to X$ that makes the following diagram commute:
    \[
    \begin{tikzcd}
        \obox{n+2}{m+2\delta} \arrow[r] \arrow[d,hook] & \obox{n+2}{m} \arrow[r] \arrow[d,hook] & X \arrow[d,"f"] \\
        \partial\gexp{I_{m+2\delta}}{n+2} \arrow[r] \arrow[urr,dotted] & \partial\gexp{I_{m}}{n+2} \arrow[r] & Y
    \end{tikzcd}
    \]
\end{proposition}
\begin{proof}
    This follows from an argument similar to the proof of \cite[Prop.~5.11]{carranza-kapulkin:cubical-graphs}.
\end{proof}

\begin{theorem} \label{thm:fib-cat-Graph-n}
    The category $\Graph$ admits a fibration category structure where
    \begin{itemize}
        \item the fibrations are maps which are sent to transferred $n$-fibrations under the nerve functor $\gnerve \from \Graph \to \cSet[n]^{\mathrm{fib}}$, and
        \item the weak equivalences are the $n$-equivalences.
    \end{itemize}
    We denote this fibration category structure by $\Graph_{n}$.
    Given this structure on $\Graph$, the nerve functor $\gnerve \from \Graph_{n} \to \cSet[n]^{\mathrm{fib}}$ is an exact functor of fibration categories.
\end{theorem}
\begin{proof}
    We verify the dual of the axioms in \cite[Def.~1.1]{szumilo-cofibration-categories}.
    The 2-out-of-6 property for $n$-equivalences, and the fact that every isomorphism is both an $n$-equivalence and an $n$-fibration, are both immediate.
    The terminal object in $\Graph$ is $I_{0}$, and for every graph $X$, the unique map $X \to I_{0}$ is a fibration, and hence an $n$-fibration.
    Since the category $\Graph$ is complete, all pullbacks exist, and are preserved by the nerve functor (\cref{thm:cubical-nerve-of-graphs}).
    Hence, the fact that the classes of fibrations and acyclic fibrations in $\Graph_{n}$ are stable under pullback follows from the corresponding property for the fibration category $\cSet[n]^{\mathrm{fib}}$. 
    Finally, \cref{thm:fib-cat-Graph-infty} tells us that every graph map factors as a weak $A$-homotopy equivalence followed by a fibration.
    Since every weak $A$-homotopy equivalence is an $n$-equivalence, and every fibration is an $n$-fibration, this also gives the required factorization as an $n$-equivalence followed by an $n$-fibration.
    This proves that $\Graph_{n}$ is a fibration category.
    The nerve functor $\gnerve \from \Graph_{n} \to \cSet[n]^{\mathrm{fib}}$ clearly preserves fibrations and acyclic fibrations.
    By \cref{thm:cubical-nerve-of-graphs}, it also preserves all pullbacks, as well as the terminal object.
    Thus, it is an exact functor of fibration categories (see \cite[Def.~1.2]{szumilo-cofibration-categories}).
\end{proof}

As an intermediate step towards the conjecture that the nerve functor $N \from \Graph \to \cSet^\mathrm{fib}$ is a weak equivalence of fibration categories, we can ask whether $N \from \Graph_{n} \to \cSet[n]^{\mathrm{fib}}$ is a weak equivalence.
We expect this problem to be more tractable, especially since there are algebraic models for homotopy $n$-types of spaces, e.g., \cite{loday:n-types-of-spaces,brown-higgins}.

For instance, for $n=1$, the fibration category $\cSet[1]^{\mathrm{fib}}$ is known to be weakly equivalent to the classical fibration category structure on the category $\Gpd$ of groupoids (see, e.g., \cite[Exercise~4.3.4]{hatcher}).

\subsection{The fundamental groupoid in discrete homotopy theory}

\begin{definition}
    Let $X$ be a graph, and $x, x' \in X$ be two vertices.
    A \emph{path} $\gamma \from x \rightsquigarrow x'$ in $X$ is a map $\gamma \from I_{\infty} \to X$ for which there exist integers $N_{-}, N_{+} \in \mathbb{Z}$ such that $\gamma\left(i\right) = x$ for all $i \leq N_{-}$ and $\gamma\left(i\right) = x'$ for all $i \geq N_{+}$.
\end{definition}

\begin{definition}
    Let $X$ be a graph, and $\gamma, \eta \from x \rightsquigarrow x'$ be two paths in $X$.
    A \emph{path-homotopy} $H \from \gamma \Rightarrow \eta$ of length $n \in \mathbb{N}$ is a map $H \from I_{\infty} \gtimes I_{n} \to X$ such that $H\left(-,0\right) = \gamma$ and $H\left(-,n\right) = \eta$, and such that each $H\left(-,i\right)$ is a path from $x$ to $x'$ in $X$.
    Two paths $\gamma, \eta \from x \rightsquigarrow x'$ are \emph{path-homotopic} if there exists a path-homotopy $H \from \gamma \Rightarrow \eta$ of any length $n \in \mathbb{N}$.
    This defines an equivalence relation on paths from $x$ to $x'$ in $X$.
    The equivalence class $[\gamma]$ of a path $\gamma$ in $X$ under path-homotopy is called its \emph{path-homotopy class}.
\end{definition}

\begin{definition}
    The \emph{discrete fundamental groupoid} $\Pi_{1}X$ of a graph $X$ is defined as follows:
    \[
    \Pi_{1}X =
    \begin{cases}
        \text{objects: } & \text{vertices } x, x', \ldots \text{ of } X \\
        \text{morphisms: } & \text{path-homotopy classes } \left[\gamma\right] \from x \to x' \text{ of} \\
        & \text{paths } \gamma \from x \rightsquigarrow x' \text{ in } X
    \end{cases}
    \]
    Composition is given by concatenating paths, i.e. $[\eta] \circ [\gamma] = [\gamma \ast \eta]$, and inverses are given by the inverse path, i.e. $[\gamma]^{-1} = [\overline{\gamma}]$.
    See \cite{kapulkin-mavinkurve} for a more detailed description of these operations.
\end{definition}

Equivalently, the discrete fundamental groupoid of a graph $X$ is the homotopy category its nerve $\gnerve X \in \Kan$.
\[
    \Pi_{1}X = \Ho (\gnerve X)
\]

\begin{theorem}\label{thm:fund-gpd-is-exact}
    The fundamental groupoid functor $\Pi_{1} \from \Graph_{1} \to \Gpd$ is an exact functor of fibration categories.
\end{theorem}

We prove this theorem in stages.
First, we prove that the fundamental groupoid functor $\Pi_{1} \from \Graph_{1} \to \Gpd$ preserves fibrations.

\begin{proposition}\label{prop:fund-gpd-preserves-fibrations}
    If a graph map $f \from X \to Y$ is a 1-fibration, then the functor $\Pi_{1}f \from \Pi_{1}X \to \Pi_{1}Y$ is an isofibration of groupoids.
\end{proposition}
\begin{proof}
    Consider an object $x \in \Pi_{1}X$ and an isomorphism $[\gamma] \from f(x) \to y$ in $\Pi_{1}Y$.
    We can pick a representative path $\gamma \from f(x) \rightsquigarrow y$ in $Y$ and consider the following lifting problem:
    \[
    \begin{tikzcd}
        \obox{1}{i,0} \arrow[r,"x"] \arrow[d,hook] & \gnerve X \arrow[d,"\gnerve f"] \\
        \cube{1} \arrow[r,"\gamma"] & \gnerve Y
    \end{tikzcd}
    \]
    Since $f$ is a 1-fibration of graphs, this lifting problem has a solution, say $\tilde{\gamma} \from \cube{1} \to \gnerve X$.
    That is, there exists a path $\tilde{\gamma} \from x \rightsquigarrow x'$ in $X$ such that $f \circ \tilde{\gamma} = \gamma$.
    Thus, there exists an isomorphism $[\tilde{\gamma}] \from x \to x'$ in $\Pi_{1}X$ such that $\Pi_{1}f [\tilde{\gamma}] = [\gamma]$.
\end{proof}

Next, we prove that the fundamental groupoid functor $\Pi_{1} \from \Graph \to \Gpd$ preserves pullbacks along 1-fibrations.

Consider an arbitrary pullback square in $\Graph$
\[
\begin{tikzcd}
    X \times_{Z} Y \arrow[r] \arrow[d] \arrow[dr,phantom,very near start,"\lrcorner"] & X \arrow[d,"f"]  \\
    Y \arrow[r,"g"] & Z
\end{tikzcd}
\]
We have the following two groupoids:
\[
    \Pi_{1}\left(X \times_{Z} Y\right) \qquad \text{and} \qquad \Pi_{1}X \times_{\Pi_{1}Z} \Pi_{1}Y
\]
They have identical objects: pairs $(x,y)$ where $x$ is a vertex in $X$ and $y$ is a vertex in $Y$ subject to the condition $f\left(x\right) = g\left(y\right)$.
However, they might a priori differ in their morphisms.

A morphism $\left(x,y\right) \to \left(x',y'\right)$ in $\Pi_{1}\left(X \times_{Z} Y\right)$ is given by the path-homotopy class $\left[\gamma\right]$ of a path $\gamma \from \left(x,y\right) \rightsquigarrow \left(x',y'\right)$ in the pullback graph $X \times_{Z} Y$, which in turn is determined by a pair $\left(\eta,\tau\right)$ where $\eta \from x \rightsquigarrow x'$ is a path in $X$ and $\eta \from y \rightsquigarrow y'$ is a path in $Y$ subject to the condition $f \circ \eta = g \circ \tau$ as paths in $Z$.

On the other hand, a morphism $\left(x,y\right) \to \left(x',y'\right)$ in $\Pi_{1}X \times_{\Pi_{1}Z} \Pi_{1}Y$ is given by a pair $\left(\left[\eta\right],\left[\tau\right]\right)$ where $\left[\eta\right]$ is the path-homotopy class of a path $\eta \from x \rightsquigarrow x'$ in $X$ and $\left[\tau\right]$ is the path-homtopy class of a path $\tau \from y \rightsquigarrow y'$ in $Y$, subject to the condition $\left[f \circ \eta\right] = \left[g \circ \tau\right]$ as morphisms in $\Pi_{1}Z$.

By the universal property of pullbacks, we have a canonical functor
\[
    \Psi \from \Pi_{1}\left(X \times_{Z} Y\right) \to \Pi_{1}X \times_{\Pi_{1}Z} \Pi_{1}Y
\]
that maps a morphism $\left[\left(\eta,\tau\right)\right]$ in $\Pi_{1}\left(X \times_{Z} Y\right)$ to a morphism $\left(\left[\eta\right],\left[\tau\right]\right)$ in $\Pi_{1}X \times_{\Pi_{1}Z} \Pi_{1}Y$.
We need to show that this functor is an isomorphism of groupoids under the additional hypothesis that $f \from X \to Z$ is a 1-fibration of graphs.

We will need the following couple of lemmas.

\begin{lemma} \label{lem:choose-nice-reps}
    Consider a pullback square in $\Graph$
    \[
    \begin{tikzcd}
        X \times_{Z} Y \arrow[r] \arrow[d] \arrow[dr,phantom,very near start,"\lrcorner"] & X \arrow[d,"f"]  \\
        Y \arrow[r,"g"] & Z
    \end{tikzcd}
    \]
    where $f \from X \to Z$ is a 1-fibration of graphs.
    Given any morphism $\left(\left[\eta\right],\left[\tau\right]\right) \from \left(x,y\right) \to \left(x',y'\right)$ in $\Pi_{1}X \times_{\Pi_{1}Z} \Pi_{1}Y$, it is always possible to choose representative paths $\eta'$ and $\tau'$ from the path-homotopy classes $\left[\eta\right]$ and $\left[\tau\right]$ respectively, such that $f \circ \eta' = g \circ \tau'$ as paths in $Z$.
\end{lemma}
\begin{proof}
    We start by choosing arbitrary representative paths $\eta$ and $\tau$ from the path-homotopy classes $\left[\eta\right]$ and $\left[\tau\right]$ respectively.
    Since the pair $\left(\left[\eta\right],\left[\tau\right]\right)$ is a morphism in $\Pi_{1}X \times_{\Pi_{1}Z} \Pi_{1}Y$, we have $\left[f \circ \eta\right] = \left[g \circ \tau\right]$ as morphisms in $\Pi_{1}Z$.
    Thus, we can also choose a path-homotopy $H \from f \circ \eta \Rightarrow g \circ \tau$ in $Z$.
    
    Consider the following lifting problem:
    \[
    \begin{tikzcd}
        \obox{1}{1,0} \arrow[d,hook] \arrow[r,"x"] & NX \arrow[d,"Nf"] \\
        \cube{1} \arrow[r,"g \circ \tau"] & NZ
    \end{tikzcd}
    \]
    Since $f \from X \to Z$ is a 1-fibration of graphs, this lifting problem admits a solution $\tilde{\tau} \from \cube{1} \to NX$.
    That is, we have a path $\tilde{\tau}$ in $X$ that starts at $x$ and satisfies $f \circ \tilde{\tau} = g \circ \tau$.
    Let $x'' \in X$ be the end-point of $\tilde{\tau}$.

    Next, we consider the following lifting problem:
    \[
    \begin{tikzcd}[column sep = large]
        \obox{2}{1,1} \arrow[d,hook] \arrow[r,"{[\eta,x,\tilde{\tau}]}"] & NX \arrow[d,"Nf"] \\
        \cube{2} \arrow[r,"H"] & NZ
    \end{tikzcd}
    \]
    Once again, since $f \from X \to Z$ is a 1-fibration of graphs, this lifting problem admits a solution $\tilde{H} \from \cube{2} \to NX$.
    That is, we have a map $\tilde{H} \from I_{\infty} \gtimes I_{\infty} \to X$ that stabilizes in all directions, and satisfies the following conditions:
    \[
        \tilde{H} \partial_{2,0} = \eta, \quad \tilde{H} \partial_{1,0} = x \degen{1}{1}, \qquad \tilde{H} \partial_{2,1} = \tilde{\tau}, \qquad f \circ \tilde{H} = H
    \]
    Let $\alpha = \tilde{H} \partial_{1,1}$.
    Then, $\alpha$ is a path from $x'$ to $x''$ in $X$, satisfying $f \circ \alpha = c_{f\left(x'\right)} = c_{g\left(y'\right)}$.

    Letting $\eta' = \tilde{\tau} \ast \overline{\alpha}$ and $\tau' = \tau$, we observe that we have a path-homotopy $\eta \Rightarrow \eta'$ in $X$, a path-homotopy $\tau \Rightarrow \tau'$ in $Y$ and furthermore, we have:
    \[
        f \circ \eta' = f \circ \left(\tilde{\tau} \ast \overline{\alpha}\right) = f \circ \tilde{\tau} \ast f \circ \overline{\alpha} = g \circ \tau \ast c_{g\left(y'\right)} = g \circ \tau = g \circ \tau'.
        \qedhere
    \]
\end{proof}

\begin{lemma} \label{lem:choice-of-reps-doesnt-matter}
    Consider a pullback square in $\Graph$
    \[
    \begin{tikzcd}
        X \times_{Z} Y \arrow[r] \arrow[d] \arrow[dr,phantom,very near start,"\lrcorner"] & X \arrow[d,"f"]  \\
        Y \arrow[r,"g"] & Z
    \end{tikzcd}
    \]
    where $f \from X \to Z$ is a 1-fibration of graphs.
    Suppose we have two paths $\eta, \eta' \from x \rightsquigarrow x'$ in $X$ and two paths $\tau, \tau' \from y \rightsquigarrow y'$ in $Y$, subject to the following conditions: $f \circ \eta = g \circ \tau$ as paths in $Z$, $f \circ \eta' = g \circ \tau'$ as paths in Z, $\left[\eta\right] = \left[\eta'\right]$ as morphisms in $\Pi_{1}X$ and $\left[\tau\right] = \left[\tau'\right]$ as morphisms in $\Pi_{1}Y$.
    Then, we have $\left[\left(\eta,\tau\right)\right] = \left[\left(\eta',\tau'\right)\right]$ as morphisms in $\Pi_{1}\left(X \times_{Z} Y\right)$.
\end{lemma}
\begin{proof}
    We start by choosing arbitrary path-homotopies $H \from \eta \Rightarrow \eta'$ in $X$ and $G \from \tau \Rightarrow \tau'$ in $Y$.
    Let $z = f\left(x\right) = g\left(y\right)$ and $z' = f\left(x'\right) = g\left(y'\right)$.
    Let $\alpha = f \circ \eta = g \circ \tau$ and $\alpha' = f \circ \eta' = g \circ \tau'$.
    Then, we can define a map $\Theta \from \bd \cube{3} \to NZ$ as follows:
    \begin{align*}
        \Theta \mid_{\image \face{3}{1,0}} &= z \degen{1}{1} \degen{2}{2}, &
        \Theta \mid_{\image \face{3}{2,0}} &= \alpha \degen{2}{2}, &
        \Theta \mid_{\image \face{3}{3,0}} &= f \circ H, \\
        \Theta \mid_{\image \face{3}{1,1}} &= z' \degen{1}{1} \degen{2}{2}, &
        \Theta \mid_{\image \face{3}{2,1}} &= \alpha' \degen{2}{2}, &
        \Theta \mid_{\image \face{3}{3,1}} &= g \circ G.
    \end{align*}
    We can also define a map $\theta \from \obox{3}{3,1} \to NX$ as follows:
    \begin{align*}
        \theta \mid_{\image \face{3}{1,0}} &= x \degen{1}{1} \degen{2}{2}, &
        \theta \mid_{\image \face{3}{2,0}} &= \eta \degen{2}{2}, &
        \theta \mid_{\image \face{3}{3,0}} &= H, \\
        \theta \mid_{\image \face{3}{1,1}} &= x' \degen{1}{1} \degen{2}{2}, &
        \theta \mid_{\image \face{3}{2,1}} &= \eta' \degen{2}{2}.
    \end{align*}
    Consider the following lifting problem:
    \[
    \begin{tikzcd}
        \obox{3}{3,1} \arrow[r,"\theta"] \arrow[d,hook] & NX \arrow[d,"Nf"]  \\
        \bd \cube{3} \arrow[r,"\Theta"] & NZ
    \end{tikzcd}
    \]
    Since $f \from X \to Z$ is a 1-fibration of graphs, this lifting problem admits a solution $\tilde{\Theta} \from \bd \cube{3} \to NX$.
    Let $H' = \tilde{\Theta} \mid_{\image \face{3}{3,1}}$.
    Then, $H'$ is a path-homotopy $\eta \Rightarrow \eta'$ that satisfies $f \circ H' = g \circ G$.
    Thus, the pair $\left(H',G\right) \from I_{\infty} \gtimes I_{\infty} \to X \times_{Z} Y$ defines a path-homotopy $\left(\eta,\tau\right) \Rightarrow \left(\eta',\tau'\right)$ in $X \times_{Z} Y$.
    It follows that $\left[\left(\eta,\tau\right)\right] = \left[\left(\eta',\tau'\right)\right]$ as morphisms in $\Pi_{1}\left(X \times_{Z} Y\right)$.
\end{proof}

We are now ready to prove that the fundamental groupoid functor $\Pi_{1} \from \Graph_{1} \to \Gpd$ preserves pullbacks along $1$-fibrations.

\begin{proposition} \label{prop:fund-gpd-preserves-pullbacks-along-fibrations}
    Consider a pullback square in $\Graph$
    \[
    \begin{tikzcd}
        X \times_{Z} Y \arrow[r] \arrow[d] \arrow[dr,phantom,very near start,"\lrcorner"] & X \arrow[d,"f"]  \\
        Y \arrow[r,"g"] & Z
    \end{tikzcd}
    \]
    where $f \from X \to Z$ is a 1-fibration of graphs.
    Then, the canonical functor
    \[
        \Psi \from \Pi_{1}\left(X \times_{Z} Y\right) \to \Pi_{1}X \times_{\Pi_{1}Z} \Pi_{1}Y
    \]
    is an isomorphism of groupoids.
\end{proposition}
\begin{proof}
    We construct a functor
    \[
        \Phi \from \Pi_{1}X \times_{\Pi_{1}Z} \Pi_{1}Y \to \Pi_{1}\left(X \times_{Z} Y\right)
    \]
    that is inverse to $\Psi$.
    Given a morphism $\left(\left[\eta\right],\left[\tau\right]\right) \from \left(x,y\right) \to \left(x',y'\right)$ in $\Pi_{1}X \times_{\Pi_{1}Z} \Pi_{1}Y$, by \cref{lem:choose-nice-reps}, we can choose representative paths $\eta'$ and $\tau'$ from the path-homotopy classes $\left[\eta\right]$ and $\left[\tau\right]$ respectively, such that $f \circ \eta' = g \circ \tau'$.
    The pair $\left(\eta',\tau'\right)$ then defines a path in the pullback graph $X \times_{Z} {Y}$.
    Let $\Phi$ map the morphism $\left(\left[\eta\right],\left[\tau\right]\right) \from \left(x,y\right) \to \left(x',y'\right)$ in $\Pi_{1}X \times_{\Pi_{1}Z} \Pi_{1}Y$ to the morphism $\left[\left(\eta',\tau'\right)\right] \from \left(x,y\right) \to \left(x',y'\right)$ in $\Pi_{1}\left(X \times_{Z} Y\right)$.
    
    By \cref{lem:choice-of-reps-doesnt-matter}, this assignment is well-defined in the sense that it is independent of the specific choice of the representative paths $\eta'$ and $\tau'$.
    Furthermore, this assignment is functorial.

    Since $\left[\eta\right] = \left[\eta'\right]$ and $\left[\tau\right] = \left[\tau'\right]$, we have $\Psi \circ \Phi = \mathrm{id}_{\Pi_{1}X \times_{\Pi_{1}Z} \Pi_{1}Y}$.

    On the other hand, given a morphism $\left[\left(\eta,\tau\right)\right] \from \left(x,y\right) \to \left(x',y'\right)$ in $\Pi_{1}\left(X \times_{Z} Y\right)$, any choice of representative path $\left(\eta,\tau\right)$ in $X \times_{Z} Y$ gives us a choice of representative paths $\eta$ in $X$ and $\tau$ in $Y$ that satisfy the condition $f \circ \eta = g \circ \tau$.
    Thus, we have $\Phi \circ \Psi = \mathrm{id}_{\Pi_{1}X \times_{\Pi_{1}Z} \Pi_{1}Y}$.
\end{proof}

\begin{proof}[Proof of \cref{thm:fund-gpd-is-exact}]
    The functor $\Pi_{1} \from \Graph_{1} \to \Gpd$ clearly preserves weak equivalences and the terminal object.
    By \cref{prop:fund-gpd-preserves-fibrations}, it preserves fibrations.
    Thus, it also preserves acyclic fibrations.
    By \cref{prop:fund-gpd-preserves-pullbacks-along-fibrations}, it preserves pullbacks along fibrations.
    Thus, it is an exact functor of fibration categories (see \cite[Def.~1.2]{szumilo-cofibration-categories}).
\end{proof}

Finally, we prove that the fundamental groupoid functor $\Pi_{1} \from \Graph \to \Gpd$ is a weak equivalence of fibration categories.

\begin{proposition}\label{prop:approximation-property-for-fund-gpd}
    Given any functor of groupoids $F \from \cat{G} \to \Pi_{1}Y$, there exists a graph map $f \from X \to Y$ such that there is a commutative square of the form:
    \[
    \begin{tikzcd}
        \cat{G}' \arrow[r,"\sim"] \arrow[d,"\sim"'] & \Pi_{1}X \arrow[d,"\Pi_{1}f"] \\
        \cat{G} \arrow[r,"F"] & \Pi_{1}Y
    \end{tikzcd}
    \]
    where $\cat{G}' \xrightarrow{\sim} \cat{G}$ and $\cat{G}' \xrightarrow{\sim} \Pi_{1}X$ are equivalences of categories.
\end{proposition}
\begin{proof}
    Suppose $\cat{G} = \coprod_{i \in \mathcal{I}}{\cat{G}_{i}}$ where each $\cat{G}_{i}$ is a connected component of $\cat{G}$.
    For each $i \in \mathcal{I}$, let $G_{i}$ be the automorphism group of some object $g_{i} \in \cat{G}_{i}$.
    By \cite[Thm.~6.5]{kapulkin-mavinkurve}, there exists a connected graph $X_{i}$ such that $A_{1}(X_{i},x_{i}) \cong G_{i}$ for any vertex $x_{i} \in X_{i}$.
    Let $X = \coprod_{i \in \mathcal{I}}{X_{i}}$, and let $\cat{G}' = \coprod_{i \in \mathcal{I}}{G_{i}}$.
    Then, we have an inclusion $\cat{G}' \hookrightarrow \cat{G}$ and a functor $\cat{G}' \to \Pi_{1}X$ that maps the object $g_{i}$ to some $x_{i} \in X_{i}$.
    Both these functors are equivalences of categories.
    For each $i \in \mathcal{I}$, we define a graph map $f_{i} \from X_{i} \to Y$ to be the constant map at the vertex $F(g_{i}) \in Y$.
    And finally, we define $f \from X \to Y$ to be $f_{i}$ on each connected component $X_{i}$.
\end{proof}

\begin{theorem}\label{thm:fund-gpd-is-a-weak-equiv}
    The fundamental groupoid functor $\Pi_{1} \from \Graph_{1\text{-}\mathrm{types}} \to \Gpd$ is a weak equivalence of fibration categories.
\end{theorem}
\begin{proof}
    Observe that $\Pi_{1} \from \Graph \to \Gpd$ reflects weak equivalences.
    The required result then follows from
    \cref{prop:approximation-property-for-fund-gpd} and \cite[Prop.~2.2]{szumilo-cofibration-categories}.
\end{proof}

\bibliographystyle{amsalphaurlmod}
\bibliography{main}

\begin{bibdiv}
\begin{biblist}

\bib{babson-barcelo-longueville-laubenbacher}{article}{
      author={Babson, Eric},
      author={Barcelo, H\'{e}l\`ene},
      author={de~Longueville, Mark},
      author={Laubenbacher, Reinhard},
       title={Homotopy theory of graphs},
        date={2006},
        ISSN={0925-9899},
     journal={J. Algebraic Combin.},
      volume={24},
      number={1},
       pages={31\ndash 44},
         url={https://doi.org/10.1007/s10801-006-9100-0},
}

\bib{brown-higgins}{article}{
      author={Brown, Ronald},
      author={Higgins, Philip~J.},
       title={The equivalence of {$\infty $}-groupoids and crossed complexes},
        date={1981},
        ISSN={0008-0004},
     journal={Cahiers Topologie G\'eom. Diff\'erentielle},
      volume={22},
      number={4},
       pages={371\ndash 386},
}

\bib{barwick-kan}{article}{
      author={Barwick, Clark},
      author={Kan, Daniel~M.},
       title={Relative categories: another model for the homotopy theory of
  homotopy theories},
        date={2012},
        ISSN={0019-3577,1872-6100},
     journal={Indag. Math. (N.S.)},
      volume={23},
      number={1-2},
       pages={42\ndash 68},
         url={https://doi.org/10.1016/j.indag.2011.10.002},
}

\bib{barcelo-kramer-laubenbacher-weaver}{article}{
      author={Barcelo, H\'{e}l\`ene},
      author={Kramer, Xenia},
      author={Laubenbacher, Reinhard},
      author={Weaver, Christopher},
       title={Foundations of a connectivity theory for simplicial complexes},
        date={2001},
        ISSN={0925-9899},
     journal={Adv. Appl. Math.},
      volume={26},
      number={1},
       pages={97\ndash 128},
         url={https://doi.org/10.1007/s10801-006-9100-0},
}

\bib{barcelo-laubenbacher}{article}{
      author={Barcelo, H\'{e}l\`ene},
      author={Laubenbacher, Reinhard},
       title={Perspectives on {$A$}-homotopy theory and its applications},
        date={2005},
        ISSN={0012-365X},
     journal={Discrete Math.},
      volume={298},
      number={1-3},
       pages={39\ndash 61},
         url={https://doi.org/10.1016/j.disc.2004.03.016},
}

\bib{brown}{article}{
      author={Brown, Kenneth~S.},
       title={Abstract homotopy theory and generalized sheaf cohomology},
        date={1973},
        ISSN={0002-9947},
     journal={Trans. Amer. Math. Soc.},
      volume={186},
       pages={419\ndash 458},
         url={https://doi.org/10.2307/1996573},
}

\bib{cisinski:presheaves}{article}{
      author={Cisinski, Denis-Charles},
       title={Les pr\'{e}faisceaux comme mod\`eles des types d'homotopie},
        date={2006},
        ISSN={0303-1179},
     journal={Ast\'{e}risque},
      number={308},
       pages={xxiv+390},
}

\bib{cisinski2014univalent}{unpublished}{
      author={Cisinski, Denis-Charles},
       title={Univalent universes for elegant models of homotopy types},
        date={2014},
        note={preprint},
}

\bib{cisinski-book}{book}{
      author={Cisinski, Denis-Charles},
       title={Higher categories and homotopical algebra},
      series={Cambridge Studies in Advanced Mathematics},
   publisher={Cambridge University Press},
        date={2019},
}

\bib{carranza-kapulkin:homotopy-groups}{article}{
      author={Carranza, Daniel},
      author={Kapulkin, Krzysztof},
       title={Homotopy groups of cubical sets},
        date={2023},
        ISSN={0723-0869},
     journal={Expositiones Mathematicae},
      volume={41},
      number={4},
       pages={125518},
  url={https://www.sciencedirect.com/science/article/pii/S0723086923000944},
}

\bib{carranza-kapulkin:cubical-graphs}{article}{
      author={Carranza, D.},
      author={Kapulkin, K.},
       title={Cubical setting for discrete homotopy theory, revisited},
        date={2024},
        ISSN={0010-437X,1570-5846},
     journal={Compos. Math.},
      volume={160},
      number={12},
       pages={2856\ndash 2903},
         url={https://doi.org/10.1112/S0010437X24007486},
}

\bib{carranza-kapulkin-kim}{article}{
      author={Carranza, Daniel},
      author={Kapulkin, Krzysztof},
      author={Kim, Jinho},
       title={Nonexistence of colimits in naive discrete homotopy theory},
        date={2023},
        ISSN={0927-2852,1572-9095},
     journal={Appl. Categ. Structures},
      volume={31},
      number={5},
       pages={Paper No. 41, 6},
         url={https://doi.org/10.1007/s10485-023-09746-9},
}

\bib{doherty-kapulkin-lindsey-sattler}{article}{
      author={Doherty, Brandon},
      author={Kapulkin, Krzysztof},
      author={Lindsey, Zachery},
      author={Sattler, Christian},
       title={Cubical models of {$(\infty,1)$}-categories},
        date={2024},
        ISSN={0065-9266,1947-6221},
     journal={Mem. Amer. Math. Soc.},
      volume={297},
      number={1484},
       pages={v+110},
         url={https://doi.org/10.1090/memo/1484},
}

\bib{Elvira-Donazar_Hernandez-Paricio_1995}{article}{
      author={Elvira-Donazar, Carmen},
      author={Hernandez-Paricio, Luis-Javier},
       title={Closed model categories for the n-type of spaces and simplicial
  sets},
        date={1995},
     journal={Mathematical Proceedings of the Cambridge Philosophical Society},
      volume={118},
      number={1},
       pages={93–103},
}

\bib{goerss-jardine}{book}{
      author={Goerss, Paul~G.},
      author={Jardine, John~F.},
       title={Simplicial homotopy theory},
      series={Modern Birkh\"auser Classics},
   publisher={Birkh\"auser Verlag, Basel},
        date={2009},
        ISBN={978-3-0346-0188-7},
         url={https://doi.org/10.1007/978-3-0346-0189-4},
        note={Reprint of the 1999 edition [MR1711612]},
}

\bib{hatcher}{book}{
      author={Hatcher, Allen},
       title={Algebraic topology},
   publisher={Cambridge University Press, Cambridge},
        date={2002},
        ISBN={0-521-79160-X; 0-521-79540-0},
}

\bib{hirschhorn}{book}{
      author={Hirschhorn, Philip~S.},
       title={Model categories and their localizations},
      series={Mathematical Surveys and Monographs},
   publisher={American Mathematical Society, Providence, RI},
        date={2003},
      volume={99},
        ISBN={0-8218-3279-4},
         url={https://doi.org/10.1090/surv/099},
}

\bib{hess-kedziorek-riehl-shipley}{article}{
      author={Hess, Kathryn},
      author={K\c{e}dziorek, Magdalena},
      author={Riehl, Emily},
      author={Shipley, Brooke},
       title={A necessary and sufficient condition for induced model
  structures},
        date={2017},
     journal={Journal of Topology},
      volume={10},
      number={2},
       pages={324\ndash 369},
  url={https://londmathsoc.onlinelibrary.wiley.com/doi/abs/10.1112/topo.12011},
}

\bib{hovey}{book}{
      author={Hovey, Mark},
       title={Model categories},
      series={Mathematical Surveys and Monographs},
   publisher={American Mathematical Society, Providence, RI},
        date={1999},
      volume={63},
        ISBN={0-8218-1359-5},
}

\bib{jardine:categorical-homotopy}{article}{
      author={Jardine, J.~F.},
       title={Categorical homotopy theory},
        date={2006},
        ISSN={1532-0073,1532-0081},
     journal={Homology Homotopy Appl.},
      volume={8},
      number={1},
       pages={71\ndash 144},
         url={http://projecteuclid.org/euclid.hha/1140012467},
}

\bib{kramer-laubenbacher}{incollection}{
      author={Kramer, Xenia~H.},
      author={Laubenbacher, Reinhard~C.},
       title={Combinatorial homotopy of simplicial complexes and complex
  information systems},
        date={1998},
   booktitle={Applications of computational algebraic geometry ({S}an {D}iego,
  {CA}, 1997)},
      series={Proc. Sympos. Appl. Math.},
      volume={53},
   publisher={Amer. Math. Soc., Providence, RI},
       pages={91\ndash 118},
         url={https://doi.org/10.1090/psapm/053/1602359},
}

\bib{kapulkin-mavinkurve}{article}{
      author={Kapulkin, Krzysztof},
      author={Mavinkurve, Udit},
       title={The fundamental group in discrete homotopy theory},
        date={2025},
        ISSN={0196-8858,1090-2074},
     journal={Adv. in Appl. Math.},
      volume={164},
       pages={Paper No. 102838, 56},
         url={https://doi.org/10.1016/j.aam.2024.102838},
}

\bib{kapulkin-voevodsky}{article}{
      author={Kapulkin, Krzysztof},
      author={Voevodsky, Vladimir},
       title={A cubical approach to straightening},
        date={2020},
        ISSN={1753-8416},
     journal={J. Topol.},
      volume={13},
      number={4},
       pages={1682\ndash 1700},
         url={https://doi.org/10.1112/topo.12173},
}

\bib{loday:n-types-of-spaces}{article}{
      author={Loday, Jean-Louis},
       title={Spaces with finitely many nontrivial homotopy groups},
        date={1982},
        ISSN={0022-4049,1873-1376},
     journal={J. Pure Appl. Algebra},
      volume={24},
      number={2},
       pages={179\ndash 202},
         url={https://doi.org/10.1016/0022-4049(82)90014-7},
}

\bib{quillen}{book}{
      author={Quillen, Daniel~G.},
       title={Homotopical algebra},
      series={Lecture Notes in Mathematics},
   publisher={Springer-Verlag, Berlin},
        date={1967},
      volume={43},
        ISBN={978-3-540-03914-3},
}

\bib{radulescu-banu}{unpublished}{
      author={R\v{a}dulescu-Banu, Andrei},
       title={Cofibrations in homotopy theory},
        date={2006},
}

\bib{szumilo-cofibration-categories}{article}{
      author={Szumi{\l}o, Karol},
       title={Homotopy theory of cofibration categories},
        date={2016},
        ISSN={1532-0073},
     journal={Homology Homotopy Appl.},
      volume={18},
      number={2},
       pages={345\ndash 357},
         url={https://doi.org/10.4310/HHA.2016.v18.n2.a19},
}

\bib{szumilo-complete-quasicategories}{article}{
      author={Szumi{\l}o, Karol},
       title={Homotopy theory of cocomplete quasicategories},
        date={2017},
        ISSN={1472-2747,1472-2739},
     journal={Algebr. Geom. Topol.},
      volume={17},
      number={2},
       pages={765\ndash 791},
         url={https://doi.org/10.2140/agt.2017.17.765},
}

\end{biblist}
\end{bibdiv}

\end{document}